\documentclass[11pt]{amsart}
\usepackage[sorted-cites]{amsrefs}
\usepackage[matrix,arrow,tips,curve]{xy}
\usepackage[english]{babel}
\usepackage{amsmath}%para ambiente matemático, por ejemplo align 
\usepackage{amssymb}
\usepackage{comment}%para hacer comentarios con \begin{comment}
\numberwithin{equation}{section}
\usepackage{ mathrsfs } % permite o uso de letras trabalhadas por ejemplo usar $$\mathscr{L}$$.
\theoremstyle{plain}
\usepackage{enumerate}
\usepackage{graphicx}
\usepackage{hyperref}
\usepackage{color}
\DeclareMathAlphabet{\pazocal}{OMS}{zplm}{m}{n}
\usepackage{lipsum}

\usepackage{xcolor}
\newtheorem{theorem}{Theorem} [section]
\newtheorem{lemma}{Lemma}[section]
\newtheorem{proposition}{Proposition}[section]

\newtheorem{cor}{Corollary}[section]

\newtheorem{example}{Example}[section]
\theoremstyle{remark}
\newtheorem{remark}{Remark}[section]

\newcommand{\func}[3]{#1:#2\longrightarrow #3}

\hypersetup{pdfpagemode=UseNone} %desaparece los marcadores del pdf en lectura y no son necesarios
\hypersetup{pdfstartview=FitH}    %para agrandar el pdf en lectura despues de que hacemos correr

\begin{document}

	\title[Some Rigidity results for $\lambda$-self-expander]{   Some rigidity properties for  $\lambda$-self-expanders }

	\subjclass[2000]{Primary: 53C42;
		Secondary: 58J50}

	\thanks{The first author is supported by CAPES of Brasil.  The second author  is partially supported by CNPq and Faperj of Brazil.}
	
	\address{Instituto de Matem\'atica e Estat\'\i stica, Universidade Federal Fluminense,
		Niter\'oi, RJ 24020, Brazil}

	\author[Saul Ancari]{Saul Ancari}

	\email{sa\_ancari@id.uff.br}
	
	\author[Xu Cheng]{Xu Cheng}
	
	\email{xucheng@id.uff.br}

	\begin{comment}
	\author[Saul Ancari]{Saul Ancari}
	
	\address{\sc Saul Ancari\\
	Instituto de Matem\'atica e Estat\'istica, Universidade Federal Fluminense, Campus Gragoat\'a, Rua Alexandre Moura 8 - S\~ao Domingos\\
	24210-200 Niter\'oi, Rio de Janeiro\\ Brazil}
	\email{samatfis@gmail.com}
	\end{comment}

	%Lo que sigue se imprimira solo en la primera hoja y en la parte inferior
	%%%%%%%%%%%%%%%%%%%%%%%%%%%%%%%%%%%%%%%%%%%%%%%%%%%%%%%%%%%%%%%%%%%%%%%%%%%%%%
	%\date{\today}
	%\subjclass[2010]{Primary 14N05, 14N15, 14M15; Secondary 14E05, 15A69, 15A75}
	%\keywords{Constant weighted mean curvature hypersurface}
	%%%%%%%%%%%%%%%%%%%%%%%%%%%%%%%%%%%%%%%%%%%%%%%%%%%%%%%%%%
	
	\begin{abstract} $\lambda$-self-expanders $\Sigma$  in $\mathbb{R}^{n+1}$ are the solutions of the isoperimetric problem  with respect to the same weighted area form as in the study of the self-expanders.
			In this paper,  we mainly extend the results on self-expanders which we obtained in \cite{ancari2020volum}  to $\lambda$-self-expanders. We prove some results that characterize the hyperplanes, spheres and cylinders as  $\lambda$-self-expanders.  We also discuss the area growths and the finiteness of
the weighted areas under the control of the growth of the mean curvature.

	\end{abstract}
	
	\maketitle
	%\tableofcontents
	
	%\bibliographystyle{amsalpha}
	%\bibliography{cwmcbibl}

	\section{introduction}
	
 A $\lambda$-self-expander is a    hypersurface  $\Sigma$ immersed in $\mathbb{R}^{n+1}$ whose mean curvature $H$ satisfies  the equation
\begin{align}\label{le-1}
	 H=-\frac{\langle x,{\bf n}\rangle}{2}+\lambda,
\end{align}
where $\lambda\in\mathbb{R}$ is  real number,  $x$ is the position vector in $\mathbb{R}^{n+1}$ and  ${\bf n}$ is the  outward unit normal field on $\Sigma$. Such hypersurfaces appear as solutions to the weighted isoperimetric problem of weight $e^{\frac{|x|^2}{4}}$ and  are characterized as  critical points of the weighted area functional 
\begin{align}\label{lexp1}
	\mathcal{A}(\Sigma)=\int_{\Sigma}e^{\frac{|x|^2}{4}}d\sigma,
\end{align}
 for  any compact normal variations $\func{F}{(-\varepsilon,\varepsilon)\times \Sigma}{\mathbb{R}^{n+1}}$   which  satisfies $\int_{\Sigma}\varphi e^{\frac{|x|^2}{4}}d\sigma=0$ , where  $\varphi=\langle \partial_tF(0,x),{\bf n}(x)\rangle$. \\
 
 On the other hand,  they can also be characterized as  constant weighted mean curvature hypersurfaces of weight $e^{\frac{|x|^2}{4}}$ (see Section \ref{le-preli}).\\

In the case that  $\lambda=0$,  $\lambda$-self-expanders  just are self-expanders of the mean curvature flows,  which  are  critical points of the weighted area  functional \eqref{lexp1}. \\

%{\color{blue}For} the sake of convenience, throughout this paper, we call the hypersurfaces immersed in $\mathbb{R}^{n+1}$ satisfying \eqref{le-1}     $\lambda$-self-expanders or $\lambda$-self-expanding hypersurfaces.

  In recent years, as the  models of singularity for the mean curvature flow,   self-shrinkers, self-expanders and translating solitons have been studied very much.
  See, for instance \cite{smoczyk2020self,ancari2020volum,halldorsson2012self}, \cite{huisken1990asymptotic,colding,rimoldi2014classification,le2011blow,cao2013gap} and   \cite{tasayco2017uniqueness,martin2015topology,impera2017rigidity,dung2020rigidity,ma2020bernstein,xin2015translating}, respectively. We would like to  mention that in some sense the research of self-expanders brings out the interest on $\lambda$-self-expanders. The analogous phenomenon has already happened on the  self-shrinkers  and translating solitons of MCF.    There are notions of constant Gaussian-weighted mean curvature hypersurfaces (also called $\lambda$-hypersurfaces) and  $\lambda$-translating solitons, which are linked with self-shrinkers  and translating solitons, respectively (see, for instance,  the work of   Q. Cheng and Wei  \cite{cheng2018complete},  McGonagle and Ross \cite{mcgonagle2015hyperplane}, and López in \cite{lopez2018compact}).  Various results have been obtained on $\lambda$-hypersurfaces and $\lambda$-translating solitons. See, for instance \cite{mcgonagle2015hyperplane,cheng2018complete,guang2018gap,cheng2016rigidity,ancari2020rigidity,wei2019note} and \cite{lopez2018compact,lopez2018invariant,li2020complete}, respectively. In \cite{lopez2020ruled}, López  estudied $\lambda$-self-expander surfaces  in $\mathbb{R}^3$ and  proved that if a  $\lambda$-self-expander surface is ruled surface, 
	then  it must be  a cylindrical surface.\\
	
	In this paper,  we study $\lambda$-self-expanders.  All  hyperplanes, Spheres $\mathbb{S}^{n}_r(0)$ centered at the origin of radius $r>0$, and  cylinders $\mathbb{S}^{k}_r(0)\times\mathbb{R}^{n-k}$, where $1\leq k\leq n-1$ and  $r>0$ are  examples of $\lambda$-self-expanders (see the details in Section \ref{le-preli}).   It is  well  known that there exists no closed self-expanders in $\mathbb{R}^{n+1}$ (see, for instance \cite{cao2013gap}).  But the situation  is different for $\lambda\neq 0$. 
	All spheres $\mathbb{S}^n_r(0)$ centered at the origin of radius $r>0$ are $\lambda$-self-expanders with $\lambda=\frac nr+\frac r2\geq \sqrt{2n}$ (see Example \eqref{Ex2} in Section \ref{le-preli}).\\

 Throughout the paper, the following convention of notation are used. The notation
 $A$ denotes the second fundamental form of   $\lambda$-self-expander $\Sigma$.   $\Sigma$ is called convex if its second fundamental form $A$  is negative definite, that is,  the eigenvalues $\lambda_i$ of $A$ satisfy $\lambda_i\leq 0$, where  $A(e_i)=-\bar{\nabla}_{e_i}{\bf n}=\lambda_i e_i$, $e_i$  is a local orthonormal frame.   $\Sigma$ is called mean convex if its mean curvature  $H=-\sum_{i=1}^n\lambda_i$ is nonnegative. We use   $x^{\top}$  to denote the tangent component of the position vector $x$.\\ 
	
We give the lower bound of $\lambda$  and obtain a gap theorem for  closed  $\lambda$-self-expanders as follows:
	\begin{theorem}\label{qg}
		Let  $\Sigma$ be a closed $\lambda$-self-expander immersed in $\mathbb{R}^{n+1}$. Then
		\begin{itemize}
		\item [(i)]  $\lambda\geq\sqrt{2n}$. Moreover, the sphere $\mathbb{S}^{n}_{\sqrt{2n}}(0)$ is the only closed  $\lambda$-self-expander   with $\lambda=\sqrt{2n}$. 
		\item [(ii)] If 
		\begin{align}\label{f5}
			|A|^2\leq-\frac{1}{2}+ \frac{\lambda\left( \lambda-\sqrt{\lambda^2-2n}\right) }{2n},
		\end{align}
		then $\Sigma$ is a sphere $\mathbb{S}^n_r(0)$ with $r=\lambda+\sqrt{\lambda^2-2n}$.
		\end{itemize}
	\end{theorem}
	 Theorem \ref{qg}(ii) gives a pinching condition of $|A|$ such that a compact $\Sigma$ is a sphere. This result can be compared with the following theorem proved by Guang \cite{guang2018gap} for $\lambda$-hypersurfaces.
	\begin{theorem}[\cite{guang2018gap}]
	    Let $\Sigma\subset\mathbb{R}^{n+1}$ be a closed $\lambda$-hypersurface  with $\lambda\geq0$. If $\Sigma$ satisfies
\begin{align*}
    	|A|^2\leq\frac{1}{2}+ \frac{\lambda\left( \lambda+\sqrt{\lambda^2+2n}\right) }{2n},
\end{align*}
then $\Sigma$ is a round sphere with radius $\sqrt{\lambda^2+2n}+\lambda.$
	\end{theorem}
	 We also obtain a condition on $\lambda$ and the mean curvature $H$ so that a complete $\lambda$-self-expander must be a sphere (see Theorem \ref{le-s6}). Besides, we prove the following result:

	\begin{theorem}\label{dxs}Let  $\Sigma$, $n\geq2$, be a closed mean convex  $\lambda$-self-expander immersed in $\mathbb{R}^{n+1}$. If
		\begin{align}\label{ineq}
			H\textup{tr}A^3+|A|^4\leq0,
		\end{align}
		then $\Sigma$ must be  a sphere $\mathbb{S}^n_r(0)$.
	\end{theorem}
		For a convex hypersurface,  note that 
	\begin{align*}
	    H\textup{tr}A^3+|A|^4=-\sum_{i<j}^{n}\lambda_i\lambda_j(\lambda_i-\lambda_j)^2\leq0.
	\end{align*}
	 Therefore, \eqref{ineq} holds and  a consequence of Theorem \ref{dxs}  is as follows.
	\begin{cor}
	Let $\Sigma$, $n\geq2$, be a closed  convex $\lambda$-self-expander  immersed in $\mathbb{R}^{n+1}$. Then $\Sigma$ must be  a sphere $\mathbb{S}^n_r(0)$.
	\end{cor}

	Next, we study  complete  $\lambda$-self-expanders with some integrability condition on the norm of  the second fundamental form.  Now we recall some related results. In \cite[Theorem 1.4]{cheng2018},  the second author of the  present paper  and Zhou proved the uniqueness of hyperplanes through the origin for mean convex self-expander hypersurfaces  under some condition on the square of the norm of the second fundamental form.

	In \cite[Theorem 2]{cheng2018complete},   Q. Cheng and Wei proved  that a complete embedded $\lambda$-hypersurface in $\mathbb{R}^{n+1}$ with  polynomial area growth  (there exist a constant $C>0$ and $r_0>0$ such that $Area(B_r(0)\cap\Sigma)\leq Cr^\alpha$, for all $r\geq r_0$ and for some $\alpha>0$ )  satisfying  $H-\lambda\geq0$ and  $\lambda 	((H-\lambda)\textup{tr}A^3+\frac{|A|^2}{2})\leq 0$   has to be a cylinder $\mathbb{S}^l_r(0)\times\mathbb{R}^{n-l},$ $0\leq l\leq n$.  Later, the same classification theorem as  Q. Cheng and Wei was proved in  \cite[Theorem 3]{ancari2019volume} by the first author of the present paper  and  Miranda, replacing the polynomial area growth by   a weaker intrinsic property on the norm of the second fundamental form. Motivated by   the above  work,  we obtain the following result on $\lambda$-self-expanders.
	
	\begin{theorem}\label{huiss}
		Let  $\Sigma$, $n\geq2$, be  a complete   $\lambda$-self-expander immersed in $\mathbb{R}^{n+1}$. If $H-\lambda \leq 0$, $|A|\in L^2(\Sigma,e^{\frac{|x|^2}{4}}d\sigma)$ and 
		\begin{align}\label{g8_1}
			\lambda(	(H-\lambda)\textup{tr}A^3-\frac{|A|^2}{2})\geq 0,
		\end{align}
		then $\lambda\geq0$ and  $\Sigma$ must be either a hyperplane or $\mathbb{S}^n_r(0)$.
		
	\end{theorem}
	
\begin{remark}
     The equality in  \eqref{g8_1} holds if $\Sigma$ is either a hyerplane or  $\mathbb{S}^{l}_r(0)\times\mathbb{R}^{n-l}$, $1\leq l \leq n$, but  the cylinders $\mathbb{S}^{l}_r(0)\times\mathbb{R}^{n-l}$, $1\leq l \leq n-1$ do not satisfy  the property  that $|A|\in L^2(\Sigma,e^{\frac{|x|^2}{4}}d\sigma)$. This is the reason why the theorem \ref{huiss} does not characterize the cylinders $\mathbb{S}^{l}_r(0)\times\mathbb{R}^{n-l}$, $1\leq l \leq n-1$.
 \end{remark}
	Recently, Smoczyk \cite[Theorem 1]{smoczyk2020self} proved that complete  self-expanders  $\Sigma\subset\mathbb{R}^{n+1}$   with $H\neq0$ are products of self-expander curves and flat subspaces, if and only if the function $\frac{|A|^2}{H^2}$ attains a local maximum. Motivated by this result, we obtain the following result, that characterizes generalized cylinders as  $\lambda$-self-expanders.
\begin{theorem}\label{smoczyk}
 Let $\Sigma$ be a complete immersed  $\lambda$-self-expander  in $\mathbb{R}^{n+1}$ with $\lambda\neq0$. Then
\begin{itemize}
\item   [\textup{(i)}] the open set  $\{x\in\Sigma; H\neq \lambda\}$  must be non-empty. 

 \item [\textup{(ii)}] Further,  if the following conditions  are satisfied  on the non-empty set $\{x\in\Sigma; H\neq \lambda\}$,
 \begin{itemize}
     \item [\textup{(a)}] $\frac{\lambda}{H-\lambda}\left((H-\lambda)\textup{tr}A^3-\frac{|A|^2}{2}\right)\leq0,$
     \item  [\textup{(b)}] the function $\frac{|A|^2}{(H-\lambda)^2}$ attains a local maximum,
 \end{itemize}
then $\Sigma$ is either a hyperplane or a cylinder $\mathbb{S}_r^{l}(0)\times\mathbb{R}^{n-l}$, $1\leq l\leq n$.
\end{itemize}
\end{theorem}

	We also prove that
	\begin{theorem}\label{le-tu}
		Let $\Sigma$ be a complete properly immersed $\lambda$-self-expander hypersurface  in $\mathbb{R}^{n+1}$. If there exists $\alpha>0$ such that
		\begin{align}\label{le-se1}
			|A|^2(H-\lambda)H+\frac{1}{2}H^2\left( 1+\frac{(\alpha+1)^2}{8}|x^{\top}|^2 \right)\leq 0,
		\end{align}
		 then   $\Sigma$ must be either a hyperplane or  $\mathbb{S}^n_r(0)$.
		
	\end{theorem}

 Using the  Simons' type equations, we prove Proposition \ref{prop-1}, which states that  if  a complete   immersed  $\lambda$-self-expander  in $\mathbb{R}^{n+1}$ has  constant mean curvature, then it  must be either a hyperplane or a cylinder   $\mathbb{S}^l_r(0)\times \mathbb{R}^{n-l}$, $1\leq l\leq n$.\\

In the last Section  \ref{section7}, we mainly extend some of  our previous results on self-expanders in  \cite{ancari2020volum} to the $\lambda$-self-expanders.  Theorem \ref{le-ineq5}  deals with finitiness of weighted areas and area growth estimate of $\lambda$-self-expanders.  Theorems \ref{le-e18} and \ref{th-74}    characterize the generalized cylinders as $\lambda$-self-expanders. \\

  It is interesting to see that unlike self-expanders, hyperplanes not through the origin  or a cylinder   $\mathbb{S}^l_r(0)\times \mathbb{R}^{n-l}$, $1\leq l\leq n$ appear in the corresponding conclusions.\\

 The rest of the paper is organized as follows: In Section \ref{le-preli}, we   give some notations and examples of $\lambda$-self-expanders. In Section \ref{simons}, we give  some  Simons' type equations for $\lambda$-self-expanders.  In Section \ref{le-section3}, we prove Theorems \ref{qg} and \ref{dxs}.  In Section \ref{le-section4}, we prove Theorems \ref{huiss} and \ref{le-tu}.  In Section \ref{le-section5}, we prove Theorem \ref{smoczyk} and Propostion \ref{prop-1}. In Section \ref{section7}, we prove  Theorems  \ref{le-ineq5}, \ref{le-e18} and \ref{th-74}.

\section{Preliminaries}\label{le-preli}
In this section, we will recall some concepts and basic facts.\\

Let $\Sigma$  denote a hypersurface immersed  in the Euclidean space $(\mathbb{R}^{n+1}, g_0)$ with the induced metric $g$. We will denote by  $d\sigma$ the area form  of $\Sigma$. In this paper, unless otherwise specified, the notations with a bar,  for instance $\bar{\nabla}$ and  $\bar{\nabla}^2$,  denote the quantities corresponding the Euclidean metric $g_0$ on $\mathbb{R}^{n+1}$.  On the other hand, the notations  like  $\nabla, \Delta$ denote the quantities corresponding the intrinsic metric ${g}$ on $\Sigma$.\\

The  isometric immersion $i: (\Sigma, g) \to (\mathbb{R}^{n+1},g_0)$ is said to be properly immersed if,  for any compact subset $\Omega$ in $\mathbb{R}^{n+1}$, the pre-image $i^{-1}(\Omega)$ is compact in $\Sigma$.\\

%A smooth metric measure space is a triple $\left(M,\bar{g},e^{-f}dv\right)$ with  a weighted volume form $e^{-f}dv$ on $M$. Here $dv$ denote the volume element of $M$ induced by the  metric $\bar{g}$. \\

%Then the function $f$ restricted on $\Sigma$, still denoted by $f$,  induces a weighted volume element  $e^{-f}d\sigma$ on $\Sigma$ and thus   a smooth metric measure space $(\Sigma, g, e^{-f}d\sigma)$. 

%and $f$ is a smooth function on $M$

Let $A$ denote the second fundamental form of $(\Sigma,g)$.  At $p\in \Sigma$, 
$A(X)= -\bar{\nabla}_X{\bf n},$
where $X\in T_p\Sigma$ and ${\bf n}$  is the  outward unit normal field on $\Sigma$. The mean curvature  $H$ of $\Sigma$  is defined as 
the trace of  $-A$.  \\

 Given a smooth function $f$ on $\mathbb{R}^{n+1}$,  it induces a smooth measure $e^{-f}d\sigma$ on $(\Sigma,g)$. Define  the weighted mean curvature  $H_f$ of $\Sigma$ of weight $e^{-f}$   by $$H_f:=H-\langle \bar{\nabla} f,{\bf n}\rangle.$$

%Given a smooth function $f$ on $\mathbb{R}^{n+1}$, $(\Sigma,g)$ induces the following  weighted hypersurface  $(\Sigma, g, e^{-f})$. The weighted mean curvature  $H_f$ of $\Sigma$ of weight $e^{-f}$ is defined  by $H_f:=H-\langle \bar{\nabla} f,{\bf n}\rangle $.  

$\Sigma$  is called constant weighted mean curvature (or simply by CWMC) hypersurface  (of weight $e^{-f}$)  if    $ H_f=\lambda$ for some $\lambda\in \mathbb{R}$, or equivalently if it satisfies
\begin{equation} \label{le-cwmc} H=\langle \bar{\nabla} f,{\bf n}\rangle+\lambda.
\end{equation}

If   $\lambda=0$, $\Sigma$  is called $f$-minimal.  There are  very interesting examples of $f$-minimal hypersurfaces:

\begin{example}\label{example1} If $f=\frac{|x|^2}4$,  $-\frac{|x|^2}4$, and  $-\left<x,w\right>$ respectively,  where $w\in \mathbb{R}^{n+1}$ is a constant vector, a $f$-minimal hypersurface $\Sigma$ is a  self-shrinker, self-expander and translator for MCF in the Euclidean space $\mathbb{R}^{n+1}$ respectively. 
\end{example}

 Taking $f=-\frac{|x|^2}{4}$ in \eqref{le-cwmc},  a CWMC hypersurface $\Sigma$ of weight $e^{\frac{|x|^2}{4}}$ in $\mathbb{R}^{n+1}$ is just a  $\lambda$-self-expander  defined by \eqref{le-1}, 
that is, $\Sigma$ satisfies the equation
\begin{align}\label{expander-1}
	H=-\frac{\langle x,{\bf n}\rangle}{2}+\lambda. 
\end{align}

Here are some examples of  $\lambda$-self-expanders.

\begin{example}
	Any hyperplane in $\mathbb{R}^{n+1}$ is a $\lambda$-self-expander  with $ \lambda = \pm\frac{d}{2}$, where $d$ denotes the distance from the origin  to the hyperplane  and the sign depends on the orientation. Indeed, let	$\Sigma \subset \mathbb{R}^{n+1}$ be a hyperplane. Let ${\bf n}$ be the  unit normal  of $\Sigma$. Then  $\Sigma = \{x \in\mathbb{R}^{n+1};\langle x, {\bf n}\rangle=d\}$ or $\Sigma = \{x \in\mathbb{R}^{n+1};\langle x, {\bf n}\rangle=-d\}$. In both cases,
	we have 
	\begin{align*}
		H=-\frac{\langle x, {\bf n}\rangle}{2}+\lambda.
	\end{align*}
	This implies that $\Sigma$ is a $\lambda$-self-expander with $\lambda = 
	 \frac{d}{2}$ or $ -\frac{d}{2}$.
	 
\end{example}
\begin{example}\label{Ex2}
	All spheres $\mathbb{S}^n_r(0)$ centered at the origin of radius $r>0$ are $\lambda$-self-expanders with $\lambda=\frac nr+\frac r2\geq \sqrt{2n}$. For each $\lambda>\sqrt{2n}$, there are two spheres $S^n_r(0)$ of radius  $r=\lambda\pm \sqrt{\lambda^2-2n} $ as $\lambda$-self-expanders with this $\lambda$. 
	
	Moreover, a sphere $\mathbb{S}^n_r(p)$ is $\lambda$-self-expander if and only if its center $p$ is the origin. In fact, let ${\bf {n}}$ be the outward unit normal  of the sphere $\mathbb{S}^n_r(p)$. We have $H=\frac nr$, $ \langle x-p, {\bf{n}}\rangle=r$, and
	hence $H+\frac{\langle x, {\bf n}\rangle}{2}=\frac nr+\frac r2-\frac{\langle p, \bf{n}\rangle}{2}=\lambda.$   Therefore $\langle p, \bf{n}\rangle$ must be  constant and thus $p$ must be the origin.

\end{example}
\begin{example}
	The cylinders $\mathbb{S}^{k}_r(0)\times\mathbb{R}^{n-k}$, where $1\leq k\leq n-1$ and  $r>0$, are also  $\lambda$-self-expanders with $\lambda=\frac kr+\frac r2\geq \sqrt{2k}$. Each $\lambda>\sqrt{2k}$ corresponds  two cylinders $\mathbb{S}^{k}_r(0)\times\mathbb{R}^{n-k}$ of radius $r=\lambda\pm \sqrt{\lambda^2-2k} $.
	
	Analogous to the argument in Example \ref{Ex2}, the only cylinders $\mathbb{S}^{k}_r(p)\times\mathbb{R}^{n-k}$  as $\lambda$-self-expanders are $\mathbb{S}^{k}_r(0)\times\mathbb{R}^{n-k}$.
	
	\end{example}

\begin{remark}Here we mention that   $\lambda$-hypersurfaces defined  in \cite{cheng2018complete}  and  $\lambda$-translating solitons defined  in \cite{lopez2018compact} are CWMC hypersurfaces  of weight  $e^{-\frac{|x|^2}{4}}$ and  $e^{\left<x,w\right>}$ respectively,  where $w\in \mathbb{R}^{n+1}$ is a constant vector.
\end{remark}

Now we state the equivalent characterization of  CWMC hypersurface.\\

The weighted area  of a measurable subset  $S\subset \Sigma$ with respect  to  the  weigth $e^{-f}$ is defined by
\begin{equation}\label{le-notation-eq-vol}\mathcal{A}_f(S):=\int_S e^{-f}d\sigma.
\end{equation}

It is known that  an $f$-minimal hypersurface  is a critical point of the weighted area functional defined in (\ref{le-notation-eq-vol}). 
%On the other hand,  it is also  a minimal hypersurface under the conformal metric $\tilde{g}=e^{-\frac{2}{n}f}g_0$ on $\mathbb{R}^{n+1}$  (see, e.g. \cite{cheng2015stability}, \cite{cheng2014eigenvalue}). 
In  \cite{mcgonagle2015hyperplane},  McGonagle-Ross  proved that a CWMC  hypersurface  $\Sigma$ of weight $e^{-f}$ is also the   critical points of the weighted area functional \eqref{le-notation-eq-vol}
but for compact normal variations $\func{F}{(-\varepsilon,\varepsilon)\times \Sigma}{\mathbb{R}^{n+1}}$   that preserve weighted volume, i.e. for variations $F$ which satisfies $\int_{\Sigma}\varphi e^{-f}d\sigma=0$ , where  $\varphi=\langle \partial_tF(0,x),{\bf n}(x)\rangle$. \\

On the  smooth metric measure space  $(\Sigma, g, e^{-f})$, there is a very important second-order elliptic operator: the drifted Laplacian
$
\Delta_{f}=\Delta-\left\langle \nabla f,\nabla\cdot\right\rangle.
$
It is well known that $\Delta_f$ is a densely defined self-adjoint
operator in $L^2(\Sigma, e^{-f}d\sigma)$, i.e. for 
$u$ and $v$ in $C^{\infty}_0(\Sigma)$, it holds that 
\begin{equation}
	\int_{\Sigma}(\Delta_{f}u) ve^{-f}d\sigma=-\int_{\Sigma}\left\langle \nabla u,\nabla v\right\rangle e^{-f}d\sigma.
\end{equation}

%Now we  give especial notations on  $\lambda$-self-expanders.
%In the following, unless otherwise specified, 
 In particular, if $\Sigma$ is a $\lambda$-self-expander, we denote by  $\mathscr{L}$ the corresponding  drifted Laplacian on $\Sigma$,  that is,  $\mathscr{L}=\Delta+\frac12\left<x,\nabla\cdot\right>$.

\section{Simons' type equations}\label{simons}

 In this section, we give some  
Simons' type equations for $\lambda$-self-expanders.

\begin{theorem}\label{le-simo}
	If $\Sigma$ is a $\lambda$-self-expander  immersed in $\mathbb{R}^{n+1}$, then
	\begin{align}\label{e2}
		&\mathscr{L}H=-|A|^2(H-\lambda)-\frac{H}{2},\\\label{ee5}
		&\mathscr{L}A=-\frac{1}{2}A-|A|^2A-\lambda A^2,\\\label{e3}
		&\mathscr{L}|A|^2=2|\nabla A|^2-|A|^2-2|A|^4-2\lambda \textup{tr}A^3,\\\label{e4}
		&\mathscr{L}|A|=\frac{|\nabla A|^2-|\nabla |A||^2}{|A|}-\frac{\lambda \textup{tr}A^3}{|A|}-|A|^3-\frac{|A|}{2}.
	\end{align}
\end{theorem}
\begin{proof}
	Let us fix a point $p\in \Sigma$, and choose a local orthonormal frame ${e_i}$,  $i=1, \ldots, n$,  for $\Sigma$ such that $\nabla_{e_i}e_j(p)=0$.  Let $h_{ij}=\langle \bar{\nabla}_{e_i}e_j, {\bf n}\rangle$. At this point $p$  we have
	\begin{align}\nonumber
		Hess(H)(e_i,e_j)&=\nabla_{e_j}\nabla_{e_i}(-\frac{\langle x,{\bf n}\rangle}{2}+\lambda)\\\nonumber
		&=-\frac{1}{2}e_j\langle x,-h_{ik}e_k\rangle\\\label{s1}
		&=-\frac{1}{2}\left( -h_{ij}-\langle x,h_{ijk}e_k\rangle-\langle x,h_{ik}h_{jk}{\bf n}\rangle \right)\\\nonumber
		&= \frac{1}{2}\langle A(e_i),e_j\rangle+\frac{1}{2}\langle(\nabla_{x^{\top}}A)(e_i),e_j\rangle+\frac{1}{2}\langle x,{\bf n}\rangle\langle A^2(e_i),e_j\rangle\\\nonumber
		&=\frac{1}{2}\langle A(e_i),e_j\rangle+\frac{1}{2}\langle(\nabla_{x^{\top}}A)(e_i),e_j\rangle-(H-\lambda)\langle A^2(e_i),e_j\rangle.
	\end{align}
	This implies that
	\begin{align}
		\Delta H=-\frac{1}{2}H -\frac{1}{2}\langle x,\nabla H\rangle-(H-\lambda)|A|^2.
	\end{align}
	Therefore
	\begin{align*}
		\mathscr{L}H=-|A|^2(H-\lambda)-\frac{H}{2}.
	\end{align*}
	So we have proved  \eqref{e2}.\\

	In order to prove \eqref{ee5}, recall the Simons' equation, that is 
	\begin{align}\label{ste}
	    \Delta A=-Hess(H)-HA^2-|A|^2A.
	\end{align}
	Combining \eqref{s1} with \eqref{ste} yields 
	\begin{align}\nonumber
	    \Delta A=-\frac{1}{2}\nabla_{x^{\top}}A-\frac{1}{2}A-|A|^2A-\lambda A^2.
	\end{align}
	Hence 
	\begin{align}
	    \mathscr{L} A=-\frac{1}{2}A-|A|^2A-\lambda A^2,
	\end{align}
	 which is \eqref{ee5}. For \eqref{e3}, we have that 
	\begin{align}\nonumber
	    \mathscr{L}|A|^2&=2\langle \mathscr{L}A,A\rangle+2|\nabla A|^2\\\nonumber
	    &=2|\nabla A|^2-|A|^2-2|A|^4-2\lambda\langle A^2,A\rangle.
	\end{align}
Further $\langle A^2,A\rangle=\textup{tr}A^3$. Therefore 
\begin{align*}
    \mathscr{L}|A|^2=2|\nabla A|^2-|A|^2-2|A|^4-2\lambda\textup{tr}A^3,
\end{align*}
	 which  is just \eqref{e3}.\\
	 \eqref{e4} follows  from  \eqref{e3} and the  following identity 
	\begin{align*}
		\mathscr{L}|A|^2=2|A|\mathscr{L}|A|+2|\nabla |A||^2.
	\end{align*}
\end{proof}

 For $\alpha\in \mathbb{R}$,  we define the operator $\mathcal{L}_\alpha=\Delta-\frac{\alpha}{2}\langle x, \nabla \cdot\rangle$  on the $\lambda$-self-expander  $\Sigma$. We get the following equations.
\begin{cor}
	Let $\Sigma$ be a  $\lambda$-self-expander immersed in $\mathbb{R}^{n+1}$.  Then, for $\alpha\in \mathbb{R}$ it holds that
	\begin{align}\label{spn7}
		&\mathcal{L}_\alpha H=-|A|^2(H-\lambda)-\frac{1}{2}H-\frac{\alpha+1}{2}\langle x,\nabla H\rangle,\\\label{le-spn8}
		&\mathcal{L}_\alpha H=-|A|^2(H-\lambda)-\frac{1}{2}H-\frac{\alpha+1}{4}A(x^{\top},x^{\top}),
	\end{align}
and 
\begin{align}\label{au1}
		\mathcal{L}_\alpha H^2=&-2|A|^2(H-\lambda)H-H^2+2\left|\nabla H-\frac{\alpha+1}{4}Hx^{\top}\right|^2\\\nonumber
	&-\frac{(\alpha+1)^2}{8}H^2|x^{\top}|^2.
\end{align}
\end{cor}
 \begin{proof}
	Since  $\mathscr{L}H=-|A|^2(H-\lambda)-\frac{H}{2}$,
	\begin{align*}
		\mathcal{L}_\alpha H=-|A|^2(H-\lambda)-\frac{1}{2}H-\frac{\alpha+1}{2}\langle x,\nabla H\rangle.
	\end{align*}
	Take a local orthonormal frame $\{e_i\}$, $i=1,...,n,$ for $\Sigma$. From $H=\lambda-\frac{1}{2}\langle x,{\bf n}\rangle$,
	\begin{align*}
		2\nabla_{e_i}H=&-\langle \nabla_{e_i}x, {\bf n}\rangle-\langle x,\nabla_{e_i}{\bf n}\rangle\\ 
		=&h_{ij}\langle x,e_j\rangle
	\end{align*}
	and hence 
	\begin{align*}
		\langle x,\nabla H\rangle=\langle x,e_i\rangle\nabla_{e_i}H=\frac{1}{2}h_{ij}\langle x,e_i\rangle\langle x,e_j\rangle=\frac{1}{2}A(x^{\top},x^{\top}).
	\end{align*}
	By this and Equation \eqref{spn7}, we have that 
	\begin{align*}
		\mathcal{L}_\alpha H&=-|A|^2(H-\lambda)-\frac{1}{2}H-\frac{\alpha+1}{2}\langle x,\nabla H\rangle\\
		&=-|A|^2(H-\lambda)-\frac{1}{2}H-\frac{\alpha+1}{4}A(x^{\top},x^{\top}).
	\end{align*}
On the other hand, combining \eqref{spn7} and the equality $\mathcal{L}_\alpha H^2=2H\mathcal{L}_\alpha H+2|\nabla H|^2$ we get 
\begin{align*}
	\mathcal{L}_\alpha H^2=&-2|A|^2(H-\lambda)H-H^2-(\alpha+1)H\langle x,\nabla H\rangle+2|\nabla H|^2\\
	=&-2|A|^2(H-\lambda)H-H^2+2\left|\nabla H-\frac{\alpha+1}{4}Hx^{\top}\right|^2\\
	&-\frac{(\alpha+1)^2}{8}H^2|x^{\top}|^2.
\end{align*}
\end{proof}

\section{Rigidity of spheres}\label{le-section3}
In this section, we will prove Theorems \ref{qg} and  \ref{dxs} which characterize the spheres as closed $\lambda$-self-expanders. In order to prove Theorem \ref{qg}, we need the following 

\begin{lemma}\label{lem-1}Let $\Sigma$ be an immersed $\lambda$-self-expander in $\mathbb{R}^{n+1}$. Then
	\begin{align}\label{f2}
		&\mathscr{L}|x|^2=|x|^2-2\lambda\langle x,{\bf n}\rangle+2n,\\\label{f3}
		&\mathcal{L}_1|x|^2=(2H-\lambda)^2+2n-\lambda^2-|x^{\top}|^2,
	\end{align}
	where the operator $\mathcal{L}_1=\Delta-\frac{1}{2}\langle x^{\top},\nabla\cdot\rangle$.
\end{lemma}
\begin{proof}
	Recall that for any hypersurface, we have $\Delta x=-H{\bf n}$. Then  
	\begin{align}\nonumber
		\Delta |x|^2&=2\langle x, \Delta x\rangle +2|\nabla x|^2\\\label{le-comp}
		&=-2H\langle x, {\bf n}\rangle+2n.
	\end{align}
	Since $H=\lambda-\frac{\langle x,{\bf n}\rangle}{2}$, we get
	\begin{align*}
		\mathscr{L}|x|^2=\Delta |x|^2+\frac{1}{2}\langle x,\nabla |x|^2\rangle
		&=-2H\langle x, {\bf n}\rangle+2n +|x^{\top}|^2\\
		&=-2\left(\lambda-\frac{\langle x,{\bf n}\rangle}{2} \right) \langle x, {\bf n}\rangle+2n +|x^{\top}|^2\\
		&=|x|^2-2\lambda\langle x,{\bf n}\rangle+2n,
	\end{align*}
	and 
	\begin{align*}
		\mathcal{L}_1|x|^2=\Delta |x|^2-\frac{1}{2}\langle x,\nabla |x|^2\rangle=&-2H\langle x, {\bf n}\rangle+2n -|x^{\top}|^2\\
		&=4H(H-\lambda)+2n -|x^{\top}|^2\\
		&=4H^2-4\lambda H+2n-|x^{\top}|^2\\
		&=(2H-\lambda)^2+2n-\lambda^2-|x^{\top}|^2.
	\end{align*}
	
\end{proof}
 Using \eqref{f2}, we prove   Theorem \ref{qg}, that is
\begin{theorem}(Theorem \ref{qg}) \label{qg2}
		Let  $\Sigma$ be a closed $\lambda$-self-expander immersed in $\mathbb{R}^{n+1}$. Then
		\begin{itemize}
		\item [(i)]  $\lambda\geq\sqrt{2n}$. Moreover, the sphere $\mathbb{S}^{n}_{\sqrt{2n}}(0)$ is the only closed $\lambda$-self-expander  with $\lambda=\sqrt{2n}$. 
		\item [(ii)] If 
		\begin{align}\label{f5}
			|A|^2\leq-\frac{1}{2}+ \frac{\lambda\left( \lambda-\sqrt{\lambda^2-2n}\right) }{2n},
		\end{align}
		then $\Sigma$ is a sphere $\mathbb{S}^n_r(0)$ with $r=\lambda+\sqrt{\lambda^2-2n}$.
		\end{itemize}
	\end{theorem}

\begin{proof}
     Since $\Sigma$ is closed, there exists a point $p\neq 0$ where $|x|$ achieves its maximum.  At $p$, $x$ and ${\bf n}$ are in the same direction, and $H(p)\geq\frac{n}{|p|}$. This implies that $\lambda=H(p)+\frac{\langle p,{\bf n}(p)\rangle}{2}\geq\frac{n}{|p|}+\frac{|p|}{2} \geq2\frac{\sqrt{n}}{\sqrt{|p|}}\cdot \frac{\sqrt{|p|}}{\sqrt{2}}=\sqrt{2n}$. 
     
      If $\lambda=\sqrt{{2n}}$,  \eqref{f2} implies
     \begin{align}
         \mathscr{L}|x|^2\geq(|x|-\lambda)^2 -\lambda^2+2n=(|x|-\lambda)^2\geq0.
     \end{align}
     By the maximum principle,  $|x|$ must be constant. We conclude that $\Sigma=\mathbb{S}^{n}_{\sqrt{2n}}(0)$.  So (i) is confirmed.
     
    Now we prove (ii).  Since $\Sigma$ is closed, we consider the point $p$ where $|x|$ achieves its minimum. We claim that $p\neq 0$. In fact, if $p=0$,   then $H(p)=\lambda$. By \eqref{f5}, we have 
	\begin{align*}
		\lambda^2=H^2(p)\leq n|A|^2\leq-\frac{n}{2}+ \frac{\lambda\left( \lambda-\sqrt{\lambda^2-2n}\right) }{2}.
	\end{align*}	
	This implies  $$0\leq-\frac{n}{2}- \frac{\lambda\left( \lambda+\sqrt{\lambda^2-2n}\right) }{2},$$
	which is a contradiction. Therefore $p\neq0$.\\
	Now, since  $p\neq0$, it follows that either  $H(p)=\lambda+\frac{|x|(p)}{2}$ or  $H(p)=\lambda -\frac{|x|(p)}{2}$. 

We claim that  $H(p)=\lambda-\frac{|x|(p)}{2}$,  which is equivalent to say that $\langle p, {\bf n}\rangle=|x|(p)$. In fact, if  $H(p)=\lambda+\frac{|x|(p)}{2}$, by \eqref{f5} we get 
	\begin{align}\label{ineq-p}
		\left( \lambda+\frac{|x|(p)}{2}\right)^2 =H^2(p)\leq n|A|^2\leq-\frac{n}{2}+ \frac{\lambda\left( \lambda-\sqrt{\lambda^2-2n}\right) }{2}.
	\end{align}
	 By (i), we know  that  $\Sigma$ satisfies  $\lambda\geq \sqrt{2n}$. This property and \eqref{ineq-p} imply that
	$$|x|(p)\leq-\lambda-\sqrt{\lambda^2-2n},$$
	which is a contradiction.  Therefore, we  have confirmed the claim.\\
	Applying  \eqref{f5} again, we obtain 
	\begin{align}
		\left( \lambda-\frac{|x|(p)}{2}\right) ^2=H^2(p)\leq n|A|^2\leq n\left( -\frac{1}{2}+ \frac{\lambda\left( \lambda-\sqrt{\lambda^2-2n}\right) }{2n}\right),
	\end{align}
	which gives
	\begin{align}\label{l-s-exx-1}
	    |x|(p)\geq\frac{1}{2\lambda}(|x|^2(p)-2\lambda|x|(p)+2n)+\lambda+\sqrt{\lambda^2-2n}.
	\end{align}
	Further, since $|x|$ achieves its minimum at $p$, by \eqref{f2},  we have $$\mathscr{L}|x|(p)=|x|^2(p)-2\lambda|x|(p)+2n\geq0.$$ 
	Therefore,  \eqref{l-s-exx-1} implies 
	\begin{align}\label{sss10}
		\min_{\Sigma}|x|= |x|(p)\geq \lambda+\sqrt{\lambda^2-2n}.
	\end{align}
	Combining \eqref{sss10} with \eqref{f2} yields
	$$\mathscr{L}(|x|^2)\geq 0.$$
	By the maximum principle, $|x|^2=constant$ for $\Sigma$. We conclude that $\Sigma$ is $\mathbb{S}^n_{r}(0)$.
     
\end{proof}

  For complete $\lambda$-self-expanders,   \eqref{f3} in Lemma \ref{lem-1} has  the following  consequence:
\begin{theorem}\label{le-s6}
	Let  $\Sigma$ be a complete properly immersed $\lambda$-self-expander in $\mathbb{R}^{n+1}$.  If $\Sigma$ satisfies 
	\begin{align}\label{le-oga}
		\lambda^2 \geq 2n+(\lambda-2H)^2,
	\end{align}  
	then $\Sigma$ is a sphere $\mathbb{S}^n_r(0)$.
\end{theorem}

\begin{proof}
    By \eqref{f3} and hypothesis \eqref{le-oga}, we have 
    \begin{align*}
        \mathcal{L}_1|x|^2=(2H-\lambda)^2+2n-\lambda^2-|x^{\top}|^2\leq0.
    \end{align*}
    Since $\Sigma$ is proper, $|x|^2$ achieves its minimum. By the maximum principle we conclude that $\Sigma$ is $\mathbb{S}^n_{r}(0)$.
\end{proof}

\begin{remark}
The condition  \eqref{le-oga} implies that  $H$ is   bounded  by $\frac{1}{2}(|\lambda|+\sqrt{\lambda^2-2n})$. Then, the  Gaussian weighted mean curvature  $H-\frac{\langle x,{\bf n}\rangle}{2}$ of $\Sigma$ is  $2H-\lambda$ and  also bounded. In \cite{cheng2019volume}, the second author of the present paper, Vieira and Zhou  (\cite[Theorem 1.4]{cheng2019volume}) proved that if $\Sigma$ is an immersed hypersurface  in $\mathbb{R}^{n+1}$,  with bounded  Gaussian weighted mean curvature, then the polynomial area growth is equivalent to the properness of the hypersurface.  Therefore, in  Theorem \ref{le-s6}, it is possible to replace  condition of properness of the hypersurface  by the polynomial area growth condition. 
\end{remark}

Now we prove Theorem \ref{dxs}.
\begin{theorem}(Theorem \ref{dxs}) \label{dxs-1}Let  $\Sigma$, $n\geq2$, be a closed mean convex $\lambda$-self-expander immersed in $\mathbb{R}^{n+1}$. If
		\begin{align}\label{ineq-c}
			H\textup{tr}A^3+|A|^4\leq0,
		\end{align}
		then $\Sigma$ must be  a sphere $\mathbb{S}^n_r(0)$.
	\end{theorem}
\begin{proof}

	Let us consider $w=\frac{\sqrt{|A|^2+\varepsilon}}{H+\varepsilon}$, where $\varepsilon$ is a positive constant. Then 
	\begin{align}\nonumber
		\mathscr{L}(\sqrt{|A|^2+\varepsilon})&=(H+\varepsilon)\mathscr{L}w+w\mathscr{L}H+2\langle \nabla w,\nabla H\rangle\\\label{hhu}
		&=(H+\varepsilon)\mathscr{L}w+w(-|A|^2H-\frac{1}{2}H+\lambda|A|^2)+2\langle \nabla w,\nabla H\rangle.
	\end{align}
	On the other hand, by \eqref{e3}, we get 
	\begin{align*}
		\mathscr{L}(\sqrt{|A|^2+\varepsilon})&=\frac{1}{2\sqrt{|A|^2+\varepsilon}}\left( \mathscr{L}|A|^2-2\left|\nabla\sqrt{|A|^2+\varepsilon} \right|^2 \right) \\
		&=\frac{1}{\sqrt{|A|^2+\varepsilon}}\left( |\nabla A|^2-\left|\nabla\sqrt{|A|^2+\varepsilon} \right|^2 \right) \\
		&+\frac{1}{2\sqrt{|A|^2+\varepsilon}}\left( -|A|^2-2\lambda\textup{tr}A^3-2|A|^4\right)\\
		&\geq\frac{1}{2\sqrt{|A|^2+\varepsilon}}\left(-|A|^2-2|A|^4-2\lambda\textup{tr}A^3 \right). 
	\end{align*}
	Combining this  with \eqref{hhu} yields
	\begin{align*}
		(H+\varepsilon)^2w\mathscr{L}w\geq&-\frac{1}{2}|A|^2-|A|^4-\lambda\textup{tr}A^3-\frac{|A|^2+\varepsilon}{H+\varepsilon}(-|A|^2H-\frac{1}{2}H+\lambda|A|^2)\\
		&-2\langle \nabla w,\nabla H\rangle(H+\varepsilon)w\\
		\geq&\frac{-\lambda(H\textup{tr}A^3+|A|^4)}{H+\varepsilon} -\frac{\varepsilon((1+2\lambda)|A|^2+2|A|^4+2\lambda\textup{tr}A^3)}{2(H+\varepsilon)}\\
		&-2\langle \nabla w,\nabla H\rangle(H+\varepsilon)w.
	\end{align*}
	 Noting that $\Sigma$ is closed, by Theorem \ref{qg2}, $\lambda\geq \sqrt{2n}$. Further,  from \eqref{ineq-c} we deduce that  $\textup{tr}A^3\leq0$ and $|A|\leq H$. Therefore

	\begin{align*}
		(H+\varepsilon)^2w\mathscr{L}w&\geq-\frac{\varepsilon((1+2\lambda)|A|^2+2|A|^4)}{2(H+\varepsilon)}-2\langle \nabla w,\nabla H\rangle(H+\varepsilon)w\\
		&\geq-\varepsilon\left(\frac{(1+2\lambda)}{2}|A|+|A|^3\right)-2\langle \nabla w,\nabla H\rangle(H+\varepsilon)w.
	\end{align*}
	
	Using integration by parts, we obtain
	\begin{align*}
		\int_{\Sigma}(H+\varepsilon)^2|\nabla w|^2e^{\frac{|x|^2}{4}}=&-\int_{\Sigma}(H+\varepsilon)^2w\mathscr{L}(w)e^{\frac{|x|^2}{4}}\\
		&-2\int_{\Sigma}(H+\varepsilon)w\langle \nabla w,\nabla H\rangle e^{\frac{|x|^2}{4}}\\
		\leq&\varepsilon\int_{\Sigma}\left(\frac{(1+2\lambda)}{2}|A|+|A|^3\right)e^{\frac{|x|^2}{4}}.
	\end{align*}
	Therefore 
	\begin{align*}
		\lim\limits_{\varepsilon\rightarrow 0}	\int_{\Sigma}(H+\varepsilon)^2|\nabla w|^2e^{\frac{|x|^2}{4}}=0.
	\end{align*} 
	Let $p\in\mathcal{S}:=\{x\in\Sigma; \mbox{ }H(x)\neq0\}$, and $B^{\Sigma}_r(p)\subset\mathcal{S}$. Then
	\begin{align*}
		\lim\limits_{\varepsilon\rightarrow 0}	\int_{B^{\Sigma}_r(p)}(H+\varepsilon)^2|\nabla w|^2e^{\frac{|x|^2}{4}}=0.
	\end{align*}
	By the dominated convergence theorem, we conclude that $\frac{|A|}{H}$ is constant in $B^{\Sigma}_r(p)$. Since $p$ is
	arbitrary,  $\frac{|A|}{H}$  must be a constant on each connected component of $\mathcal{S}$. By \eqref{e2}, \eqref{e4} and  \eqref{ineq-c}, we conclude that $|\nabla A|=|\nabla |A||$ on $\mathcal{S}$. Now we use an argument  similar to  the one used by Huisken in \cite{huisken54local} (See, e.g. the proof of Theorem 0.17 in \cite{colding}). Since $\Sigma$ is closed, there exists $q\in\mathcal{S}$ such that rank$A=n\geq2$ and then $|\nabla A|=0$ on $\mathcal{D}$, where $\mathcal{D}$ denotes the  connected component  of $\mathcal{S}$ containing $q$.  In particular, $|A|$ and $H$ are positive constants on $\mathcal{D}$. Using a continuity argument, we have $\Sigma\setminus\mathcal{S}=\emptyset$ and thus $\mathcal{D}=\Sigma$. Hence, on $\Sigma$, $|\nabla A|=0$ and $H$ is a positive constant. By  a Lawson's result \cite[Theorem 4]{lawson}, $\Sigma$ must be  a sphere $\mathbb{S}^{n}_r(0)$. 
\end{proof} 

\section{Rigidity of hyperplanes and spheres}\label{le-section4}
In this section, we will prove some rigidity results that characterize the hyperplanes and spheres as $\lambda$-self-expanders.\\

First, we prove Theorem \ref{huiss}, which is, 
\begin{theorem}[Theorem \ref{huiss}]
		Let  $\Sigma$, $n\geq2$, be  a complete   $\lambda$-self-expander immersed in $\mathbb{R}^{n+1}$. If $H-\lambda \leq 0$, $|A|\in L^2(\Sigma,e^{\frac{|x|^2}{4}}d\sigma)$ and 
		\begin{align}\label{g8}
			\lambda(	(H-\lambda)\textup{tr}A^3-\frac{|A|^2}{2})\geq 0,
		\end{align}
		then $\lambda\geq0$ and  $\Sigma$ must be either a hyperplane or  a sphere  $\mathbb{S}^n_r(0)$.
		
	\end{theorem}

\begin{proof}
	For $\lambda\leq0$, from Lemma \ref{le-simo} and  the hypothesis $H-\lambda\leq0$ we have 
	\begin{align}\label{eq-1.4}
		\mathscr{L}(H-\lambda)=-|A|^2(H-\lambda)-\frac{1}{2}H\geq0.
	\end{align}
	Since $H-\lambda\leq0$, by the maximum principle we  have either $H-\lambda\equiv0$ or $H-\lambda< 0$.  If  $H-\lambda\equiv0$,  $\Sigma$ must be a hyperplane. \\
	
	For  $\lambda>0$ and   $H-\lambda=0$ at some point $p\in\Sigma$, from hypothesis \eqref{g8}
	\begin{align*}
		-\frac{ |A|^2}{2}=(H-\lambda)\textup{tr}A^3-\frac{|A|^2}{2}\geq 0 
	\end{align*}    
	at $p\in \Sigma$. This implies that $|A|(p)=0$,  and thus $0=H(p)=\lambda$.  But this contradicts the fact that $\lambda>0.$\\
	
	To complete the proof, we only need to consider the case $H-\lambda<0$. Define the function
	\begin{align*}
		w=\frac{\sqrt{|A|^2+\varepsilon}}{H-\lambda}.
	\end{align*} 
	Using $\eqref{e2}$, we obtain
	\begin{align*}
		\mathscr{L}\sqrt{|A|^2+\varepsilon}=&w\mathscr{L}(H-\lambda)+(H-\lambda)\mathscr{L}w+2\langle \nabla w,\nabla H\rangle\\
		=&\sqrt{|A|^2+\varepsilon}\left[ -\frac{\lambda}{2(H-\lambda)}-\bigg( |A|^2+\frac{1}{2}\bigg) \right] +(H-\lambda)\mathscr{L}w\\
		&+2\langle \nabla w,\nabla H\rangle.
	\end{align*}
	Hence
	\begin{align}\nonumber
		(H-\lambda)^2w\mathscr{L}w=&\sqrt{|A|^2+\varepsilon}\mathscr{L}\sqrt{|A|^2+\varepsilon} -2\langle \nabla w,\nabla H\rangle w(H-\lambda)\\\label{g2}
		&+(|A|^2+\varepsilon)\left( \frac{\lambda}{2(H-\lambda)}+|A|^2+\frac{1}{2}\right).
	\end{align}
	On the other hand, \eqref{e3} yields
	\begin{align}\nonumber
		\sqrt{|A|^2+\varepsilon}\mathscr{L}\sqrt{|A|^2+\varepsilon}=&\frac{1}{2}\left( \mathscr{L}|A|^2-2\left| \nabla\sqrt{|A|^2+\varepsilon}\right| ^2\right) \\\nonumber
		=& |\nabla A|^2-\left| \nabla\sqrt{|A|^2+\varepsilon}\right| ^2\\\nonumber
		&-\frac{1}{2}\left( |A|^2+2\lambda\textup{tr}A^3+2|A|^4\right).
	\end{align}
	Using the inequality $|\nabla A|^2-\left| \nabla\sqrt{|A|^2+\varepsilon}\right| ^2\geq 0$,  we obtain
	\begin{align}\label{g3}
		\sqrt{|A|^2+\varepsilon}\mathscr{L}\sqrt{|A|^2+\varepsilon}\geq  -\frac{|A|^2}{2}-\lambda\textup{tr}A^3-|A|^4.
	\end{align}
	Combining  \eqref{g2} and  \eqref{g3} it follows that 
	\begin{align}\nonumber
		(H-\lambda)^2w\mathscr{L}w\geq&\frac{\lambda}{\lambda-H}\left((H-\lambda)\textup{tr}A^3-\frac{|A|^2}{2}\right)+\varepsilon|A|^2+\frac{\varepsilon H}{2(H-\lambda)}\\\label{g4}
		&-2\langle \nabla w,\nabla H\rangle w(H-\lambda).
	\end{align}
	Since $H-\lambda < 0$ and $\lambda\left((H-\lambda)\textup{tr}A^3-\frac{|A|^2}{2}\right)\geq 0$, \eqref{g4} yields 
	\begin{align}\label{g5}
		(H-\lambda)^2w\mathscr{L}w\geq\frac{\varepsilon H}{2(H-\lambda)}-2\langle \nabla w,\nabla H\rangle w (H-\lambda).
	\end{align}
	For a nonnegative function $\varphi\in C^{\infty}_0(\Sigma)$, using integration by part and \eqref{g5}, we get 
	\begin{align*}
		\int_{\Sigma}\varphi^2(H-\lambda)^2|\nabla w|^2e^{\frac{|x|^2}{4}}=&- \int_{\Sigma}\varphi^2(H-\lambda)^2w\mathscr{L}we^{\frac{|x|^2}{4}}\\
		& -2\int_{\Sigma}\varphi^2\left\langle\nabla w, \nabla H\right\rangle w(H-\lambda) e^{\frac{|x|^2}{4}}\\
		& -2\int_{\Sigma}\left\langle\varphi(H-\lambda)\nabla w, w(H-\lambda)\nabla \varphi\right\rangle e^{\frac{|x|^2}{4}}\\
		\leq& \int_{\Sigma}\frac{\varepsilon H\varphi^2}{2(\lambda-H)}e^{\frac{|x|^2}{4}}+\frac{1}{2}\int_{\Sigma}\varphi^2(H-\lambda)^2|\nabla w|^2e^{\frac{|x|^2}{4}}\\
		&+2\int_{\Sigma}w^2(H-\lambda)^2|\nabla \varphi|^2e^{\frac{|x|^2}{4}}.
	\end{align*}
	Therefore, 
	\begin{align}\label{g7}
		\int_{\Sigma}\varphi^2(H-\lambda)^2|\nabla w|^2e^{\frac{|x|^2}{4}}\leq&\int_{\Sigma}\frac{\varepsilon H\varphi^2}{\lambda-H}e^{\frac{|x|^2}{4}}+4\int_{\Sigma}|\nabla \varphi|^2|A|^2e^{\frac{|x|^2}{4}}\\\nonumber
		&+4\varepsilon\int_{\Sigma}|\nabla \varphi|^2e^{\frac{|x|^2}{4}}.
	\end{align}

	Now, for the following set $$\mathcal{S}=\{p\in\Sigma;|A|(p)=0\},$$
	we consider two  cases separately:\\
	
	\textbf{Case} $({\mathcal{S}\neq\emptyset})$.  Since  $H-\lambda<0$ and $\mathcal{S}\neq\emptyset$,  we get $\lambda>0$. First, we claim that $\Sigma $ is noncompact. In fact, if $\Sigma$ is compact, then  taking  $\varphi=1$ in \eqref{g7}, we deduce that
	\begin{align}\nonumber
		\lim_{\varepsilon\rightarrow 0}	\int_{\Sigma}(H-\lambda)^2|\nabla w|^2e^{\frac{|x|^2}{4}}=0.
	\end{align}
	Suppose that $\mathcal{D}:=\Sigma\setminus \mathcal{S}$ is not empty. Let  $q\in\mathcal{D}$ and $B_r^{\Sigma}(q)\subset\mathcal{D}$. Then
	\begin{align}\nonumber
		\lim_{\varepsilon\rightarrow 0}	\int_{B_r^{\Sigma}(q)}(H-\lambda)^2|\nabla w|^2e^{\frac{|x|^2}{4}}=0.
	\end{align}
	By the dominated convergence theorem, we conclude that $\frac{|A|}{H-\lambda}$ is constant in $B^{\Sigma}_r(q)$.  Since $q$ is
	arbitrary, it is possible to conclude that $\frac{|A|}{H-\lambda}$ is constant on each connected component of $\mathcal{D}$. Since $\mathcal{S}\neq\emptyset$, using a continuity argument, we conclude that $|A|\equiv0$, but this contradicts the assumption that $\mathcal{D}\neq \emptyset$. Hence $\mathcal{S}=\Sigma$. In particular $\Sigma $ must be minimal, which  is a contradiction. Therefore, we have confirmed the claim.

		Now we assume that $\Sigma$ is noncompact. For a fixed point $p\in \Sigma$, choose $\varphi=\varphi_k$, where $\varphi_k$ are nonnegative cut-off functions, such that $\varphi_k=1$ on $B^{\Sigma}_k(p)$, $\varphi_k=0$ on  $\Sigma \setminus B^{\Sigma}_{2k}(p)$ and $|\nabla \varphi_k|\leq \frac{1}{k}$ for every $k$. Substitute $\varphi$ in \eqref{g7}: 
	\begin{align}\nonumber
		\int_{\Sigma}\varphi^2_k(H-\lambda)^2|\nabla w|^2e^{\frac{|x|^2}{4}}\leq& \varepsilon\int_{\Sigma}\frac{ H\varphi_k^2}{ \lambda-H}e^{\frac{|x|^2}{4}}+4\int_{\Sigma}|\nabla \varphi_k|^2|A|^2e^{\frac{|x|^2}{4}}\\\nonumber
		&+4\varepsilon\int_{\Sigma}|\nabla \varphi_k|^2e^{\frac{|x|^2}{4}}\\\nonumber
		\leq& \varepsilon\int_{B^{\Sigma}_{2k}(p)}\frac{ |H|}{\lambda-H}e^{\frac{|x|^2}{4}}\\\nonumber
		&+\frac{4}{k^2}\int_{B^{\Sigma}_{2k}(p)\setminus B^{\Sigma}_k(p)}|A|^2e^{\frac{|x|^2}{4}}
		+\frac{4\varepsilon}{k^2}\int_{B^{\Sigma}_{2k}(p)}e^{\frac{|x|^2}{4}}.
	\end{align}
	Choosing $\varepsilon=\frac{1}{k}\left( \int_{B^{\Sigma}_{2k}(p)}e^{\frac{|x|^2}{4}}+\int_{B^{\Sigma}_{2k}(p)}\frac{ |H|}{\lambda-H}e^{\frac{|x|^2}{4}}\right)^{-1} $,  we have 
	\begin{align*}
		\int_{\Sigma}\varphi^2_k(H-\lambda)^2|\nabla w_k|^2e^{\frac{|x|^2}{4}}\leq& \frac{1}{k}+\frac{4}{k^2}\int_{B^{\Sigma}_{2k}(p)\setminus B^{\Sigma}_k(p)}|A|^2e^{\frac{|x|^2}{4}}+\frac{4}{k^3}.
	\end{align*}
In the above, $w_k$ denotes $w$ in which $\varepsilon$ has been substituted.

	Since $|A|\in L^2(\Sigma,e^{\frac{|x|^2}{4}}d\sigma)$, it follows that 
	\begin{align*}
		\lim_{k\rightarrow \infty}\int_{\Sigma}\varphi^2_k(H-\lambda)^2|\nabla w_k|^2e^{\frac{|x|^2}{4}}=0.
	\end{align*}
	Given that,  $\mathcal{S}\neq\emptyset$, consider $\mathcal{D}=\Sigma\setminus \mathcal{S}$. Since $\mathcal{D}$ is an open set, let $q\in\mathcal{D}$ and $B^{\Sigma}_r(q)\subset \mathcal{D}$. For $k$ sufficiently large, $B^{\Sigma}_r(q)\subset \textup{supp}\varphi_k$ and $\varphi_k=1$ on $B^{\Sigma}_r(q)$. Hence 
	\begin{align*}
		\lim_{k\rightarrow \infty}\int_{B^{\Sigma}_r(q)}(H-\lambda)^2|\nabla w_k|^2e^{\frac{|x|^2}{4}}=0.
	\end{align*}
	By the dominated convergence theorem, we conclude that $\frac{|A|}{H-\lambda}$ is constant on $B^{\Sigma}_r(q)$. Since $q$ is arbitrary, it is possible to conclude that $\frac{|A|}{H-\lambda}$ is constant on $\mathcal{D}$. Since $\mathcal{S}\neq\emptyset$, using a continuity argument, we conclude that $|A|=0$  on $\Sigma$. We conclude that $\Sigma$ is a hyperplane.\\
	
	\textbf{Case} ($\mathcal{S}= \emptyset$). Choosing $\varepsilon =0$ in \eqref{g7}, we obtain 
	\begin{align}\label{le1}
		\int_{\Sigma}\varphi^2(H-\lambda)^2|\nabla w|^2e^{\frac{|x|^2}{4}}\leq 4\int_{\Sigma}|\nabla \varphi|^2|A|^2e^{\frac{|x|^2}{4}}.
	\end{align} 
	
	In the case $\Sigma$ is compact, choose $\varphi=1$. Otherwise, fixing a point $p\in \Sigma$, consider $\varphi=\varphi_k$, where $\varphi_k$ are nonnegative cut-off function, such that $\varphi_k=1$ on $B^{\Sigma}_k(p)$, $\varphi_k=0$ on  $\Sigma \setminus B^{\Sigma}_{2k}(p)$ and $|\nabla \varphi_k|\leq \frac{1}{k}$, then

	\begin{align}\nonumber
		\int_{\Sigma}\varphi^2_k(H-\lambda)^2|\nabla w|^2e^{\frac{|x|^2}{4}}&\leq 4\int_{\Sigma}|\nabla \varphi_k|^2|A|^2e^{\frac{|x|^2}{4}}\\\label{g6}
		&\leq \frac{4}{k^2}\int_{B^{\Sigma}_{2k}(p)\setminus B^{\Sigma}_k(p)}|A|^2e^{\frac{|x|^2}{4}}.
	\end{align}
	Since $|A|\in L^{2}(\Sigma,e^{\frac{|x|^2}{4}}d\sigma)$, letting $k\rightarrow \infty$ in \eqref{g6} and using the dominated convergence theorem, we get  
	\begin{align*}
		\int_{\Sigma}(H-\lambda)^2\left|\nabla \frac{|A|}{H-\lambda}\right|^2e^{\frac{|x|^2}{4}}=0. 
	\end{align*}
	This implies that $|A|=C(H-\lambda)$, for a constant $C<0.$  So from  \eqref{e2}, it follows that   
	\begin{align}\label{g9}
		\mathscr{L}|A|+\left(|A|^2+\frac{1}{2} \right)|A|=-\frac{\lambda |A|}{2(H-\lambda)}. 
	\end{align} 
	Combining  \eqref{e4} and \eqref{g9}, we obtain  
	\begin{align*}
		\frac{|\nabla A|^2-|\nabla|A||^2}{|A|}=\frac{\lambda}{(H-\lambda)}\left( (H-\lambda)\textup{tr}A^3-\frac{|A|^2}{2}\right). 
	\end{align*}
		Further,  by hypothesis $\lambda \left((H-\lambda)\textup{tr}A^3-\frac{|A|^2}{2}\right)\geq 0 $, and  $H-\lambda<0$. Therefore  
		$$|\nabla A|=|\nabla |A||.$$
		Now again, as in the proof of Theorem \ref{dxs}, we use an argument similar to  the Huisken's in \cite{huisken54local} ( See, e.g. the proof of Theorem 0.17 in \cite{colding}).  Following the proof in \cite{colding}, $|\nabla A|=|\nabla |A||$ implies two possible sub-cases: (I) If the rank of $A$ is greater than 2,  $\nabla A =0$.  By {Theorem 4}  in \cite{lawson}, it follows that $\Sigma$ is either $\mathbb{S}^n_r(0)$ or $\mathbb{S}^l_r(0)\times\mathbb{R}^{n-l}$, $2\leq l\leq n-1$. Since  $\mathbb{S}^l_r(0)\times\mathbb{R}^{n-l}$, $2\leq l\leq n-1$, do not satisfy the hypothesis that $|A|\in L^2(\Sigma,e^{\frac{|x|^2}{4}}d\sigma)$, we  conclude that $\Sigma$ must be  $\mathbb{S}^n_r(0)$. (II) If  the rank of $A$ is 1, then $\Sigma$ is the product of a curve $\Gamma\subset\mathbb{R}^2$ and a $(n-1)$-dimensional hyperplane, but this  contradicts the fact that  $|A|\in L^2(\Sigma,e^{\frac{|x|^2}{4}}d\sigma)$. 
%%%%%%%%%%%%%%%%%%%%%%%%%%%%%%%%%%%%%%%%%%%%%%%%%%%%%%%%%%%%%%%%%%%%%		
	
\end{proof}

Now we  prove Theorem \ref{le-tu}.
\begin{theorem}[Theorem \ref{le-tu}]
		Let $\Sigma$ be a complete properly immersed $\lambda$-self-expander   in $\mathbb{R}^{n+1}$. If there exists $\alpha>0$ such that
		\begin{align}\label{le-se1}
			|A|^2(H-\lambda)H+\frac{1}{2}H^2\left( 1+\frac{(\alpha+1)^2}{8}|x^{\top}|^2 \right)\leq 0,
		\end{align}
		then $\Sigma$ must be either a hyperplane or  $\mathbb{S}^n_r(0)$.
	\end{theorem}
\begin{proof}
	
	Let $\varphi\in C^{\infty}_0(\Sigma)$, from \eqref{au1} and hypothesis \eqref{le-se1}, we have
	\begin{align*}
		2\int_{\Sigma}\varphi^2H^{2}\left|\nabla H-\frac{\alpha+1}{4}x^{\top}H\right|^2e^{-\frac{\alpha}{4}|x|^2}
		\leq\int_{\Sigma}\varphi^2H^{2}\mathcal{L}_\alpha H^2e^{-\frac{\alpha}{4}|x|^2}.
	\end{align*}
	Further,
	\begin{align*}
		\int_{\Sigma}\varphi^2H^{2}\mathcal{L}_\alpha H^2e^{-\frac{\alpha}{4}|x|^2}=&-2\int_{\Sigma}\varphi H^{2}\langle \nabla \varphi,\nabla H^2\rangle e^{-\frac{\alpha}{4}|x|^2}\\
		&- \int_{\Sigma}\varphi^2|\nabla H^2|^2e^{-\frac{\alpha}{4}|x|^2}\\
		\leq&-\frac{1}{2}\int_{\Sigma}\varphi^2 |\nabla H^2|^2e^{-\frac{\alpha}{4}|x|^2}
		+2\int_{\Sigma}H^{4}|\nabla\varphi|^2e^{-\frac{\alpha}{4}|x|^2}.
	\end{align*}
	Therefore
	\begin{align}\label{lse2}
	\int_{\Sigma}H^{4}|\nabla\varphi|^2e^{-\frac{\alpha}{4}|x|^2}\geq&\int_{\Sigma}\varphi^2H^{2}\left|\nabla H-\frac{\alpha+1}{4}x^{\top}H\right|^2e^{-\frac{\alpha}{4}|x|^2}\\\nonumber
	&+\frac{1}{4}\int_{\Sigma}\varphi^2 |\nabla H^2|^2e^{-\frac{\alpha}{4}|x|^2}.
	\end{align}
	On the other hand, from \eqref{le-se1} it follows that $(H-\lambda)H\leq0$. In particular  $|H|\leq|\lambda|.$ Thus, by Corollary \ref{le-spn10}, which will be proved later in this paper , we  have that $\int_{\Sigma}H^4e^{-\frac{\alpha}{4}|x|^2}<\infty.$ Choose  $\varphi=\varphi_j$ in \eqref{lse2}, where  $\varphi_j$ are nonnegative cut-off functions satisfying that  $\varphi_j = 1$ on $B_j(0)$, $\varphi_j = 0$ on $\Sigma\setminus B_{j+1}(0)$ and $|\nabla \varphi_j|\leq 1$. Using  the monotone convergence theorem, it follows that
	$$H^{2}\left|\nabla H-\frac{\alpha+1}{4}x^{\top}H\right|^2=0 \mbox{ and } |\nabla H^2|^2=0.$$
	This implies that $H$ is constant and $H^{4}\left|x^{\top}\right|^2=0$. If $H\equiv0$, then  $\Sigma$ is  a hyperplane. If $H\neq 0$, then  $ |x^{\top}|=0$, $|x|^2=|\langle x,{\bf n}\rangle|^2=4(\lambda-H)^2$ is constant. Thus $\Sigma$ is  a sphere. 
\end{proof}
 \section{Rigidity of generalized cylinders}\label{le-section5}
 In this section, using the maximum principle we will prove  Theorem \ref{smoczyk}  that   characterizes the  generalized cylinders as $\lambda$-self-expander hypersurfaces.

 \begin{theorem}[Theorem \ref{smoczyk}]
 Let $\Sigma$ be a complete immersed  $\lambda$-self-expander in $\mathbb{R}^{n+1}$ with $\lambda\neq0$. Then
\begin{itemize}
\item   [\textup{(i)}] the open set  $\{x\in\Sigma; H\neq \lambda\}$  must be non-empty. 

 \item [\textup{(ii)}] Further,  if the following conditions   are satisfied  on the non-empty set $\{x\in\Sigma; H\neq \lambda\}$,
 \begin{itemize}
     \item [\textup{(a)}] $\frac{\lambda}{H-\lambda}\left((H-\lambda)\textup{tr}A^3-\frac{|A|^2}{2}\right)\leq0,$
     \item  [\textup{(b)}] the function $\frac{|A|^2}{(H-\lambda)^2}$ attains a local maximum,
 \end{itemize}
then $\Sigma$ is either a hyperplane or a cylinder $\mathbb{S}_r^{l}(0)\times\mathbb{R}^{n-l}$, $1\leq l\leq n$.
\end{itemize}
\end{theorem}

\begin{proof}
First we confirm $\textup{(i)}$.  In fact, by
contradiction, assume that  $H-\lambda\equiv 0$  on $\Sigma$. From  \eqref{e2} we conclude that $\lambda=0$, a contradiction. Therefore, the set   $\{x\in\Sigma; H\neq \lambda\}$ is non-empty.

Now we prove  $\textup{(ii)}$. By  hypothesis that the function $\frac{|A|^2}{(H-\lambda)^2}$ attains a local maximum   on the set $\mathcal{D}=\{x\in\Sigma; H\neq \lambda\}$, there is a point $p\in \mathcal{D}$  and a
connected open neighborhood $U\subset\mathcal{D}$  of $p$ such that $\frac{|A|^2}{(H-\lambda)^2}(p)=\max\limits_{U}\frac{|A|^2}{(H-\lambda)^2}$.
Computing $\mathscr{L}\frac{|A|^2}{(H-\lambda)^2}$ on $U$ we get 

\begin{align}\label{l-s-e3}
    \mathscr{L}\frac{|A|^2}{(H-\lambda)^2}=2\left\langle\mathscr{L}\frac{A}{H-\lambda},\frac{A}{H-\lambda}\right\rangle+2\left|\nabla\frac{A}{H-\lambda}\right|^2.
\end{align}
Further 
\begin{align}\label{l-s-e1}
    \mathscr{L}\frac{A}{H-\lambda}=\frac{1}{H-\lambda}\mathscr{L}A+A\mathscr{L}\frac{1}{H-\lambda}-\frac{2}{(H-\lambda)^2}\nabla_{\nabla (H-\lambda)}A,
\end{align}
and
\begin{align}\label{l-s-e2}
    \mathscr{L}\frac{1}{H-\lambda}=-\frac{\mathscr{L}(H-\lambda)}{(H-\lambda)^2}+2\frac{|\nabla (H-\lambda)|^2}{(H-\lambda)^3}.
\end{align}
Substituting \eqref{l-s-e1} and \eqref{l-s-e2} into \eqref{l-s-e3}, we obtain
\begin{align}\nonumber
    &\mathscr{L}\frac{|A|^2}{(H-\lambda)^2}\\\nonumber
    =&\frac{2}{(H-\lambda)^2}\langle\mathscr{L}A,A\rangle+\frac{2|A|^2}{H-\lambda}\left(-\frac{\mathscr{L}(H-\lambda)}{(H-\lambda)^2}+2\frac{|\nabla (H-\lambda)|^2}{(H-\lambda)^3}\right)\\\nonumber
    &-\frac{4}{(H-\lambda)^3}\langle\nabla_{\nabla (H-\lambda)}A,A\rangle+2\left|\nabla\frac{A}{H-\lambda}\right|^2\\\nonumber
    =&\frac{2}{(H-\lambda)^2}\langle\mathscr{L}A,A\rangle-\frac{2|A|^2}{(H-\lambda)^3}\mathscr{L}(H-\lambda)+\frac{4|A|^2}{(H-\lambda)^4}|\nabla (H-\lambda)|^2\\\nonumber
    &-\frac{2}{(H-\lambda)^3}\langle \nabla |A|^2,\nabla (H-\lambda)\rangle+2\left|\nabla\frac{A}{H-\lambda}\right|^2\\\nonumber
    =&\frac{2}{(H-\lambda)^2}\langle\mathscr{L}A,A\rangle-\frac{2|A|^2}{(H-\lambda)^3}\mathscr{L}(H-\lambda)\\\label{smcz-1}
    &-\left\langle \nabla \frac{|A|^2}{(H-\lambda)^2},\nabla(\log (H-\lambda)^2)\right\rangle
    +2\left|\nabla\frac{A}{H-\lambda}\right|^2.\\\nonumber
\end{align}
 From \eqref{smcz-1},  \eqref{e2} and \eqref{ee5}  it follows that 
    \begin{align}\nonumber
        \mathscr{L}\left(\frac{|A|^2}{(H-\lambda)^2}\right)=&-\left\langle\nabla\frac{|A|^2}{(H-\lambda)^2},\nabla \log(H-\lambda)^2\right\rangle+2\left|\nabla \frac{A}{H-\lambda}\right|^2\\\nonumber
        &-\frac{2\lambda}{(H-\lambda)^3}\left((H-\lambda)\textup{tr}A^3-\frac{|A|^2}{2}\right).
     \end{align}
Combining this with  the hypothesis $\textup{(a)}$, we obtain
\begin{align*}
     \mathscr{L}\left(\frac{|A|^2}{(H-\lambda)^2}\right)+\left\langle\nabla\frac{|A|^2}{(H-\lambda)^2},\nabla \log(H-\lambda)^2\right\rangle\geq0.
\end{align*}
Since $\frac{|A|^2}{(H-\lambda)^2}$ attains a local  maximum, the maximum principle implies that $\frac{|A|^2}{(H-\lambda)^2}=c$, where $c\in\mathbb{R}$ is a nonnegative constant. Since $H-\lambda\neq0$,  either $|A|=\sqrt{c}(H-\lambda)$ or  $|A|=-\sqrt{c}(H-\lambda)$. 
If $c=0$, $U$ must be contained in  a hyperplane.  If $c>0$, from  \eqref{e2}, it follows that   
	\begin{align}\label{l-g9}
		\mathscr{L}|A|+\left(|A|^2+\frac{1}{2} \right)|A|=-\frac{\lambda |A|}{2(H-\lambda)}. 
	\end{align} 
	Combining  \eqref{e4} and \eqref{l-g9}, we obtain  
	\begin{align*}
		|\nabla A|^2-|\nabla|A||^2=\frac{\lambda}{(H-\lambda)}\left( (H-\lambda)\textup{tr}A^3-\frac{|A|^2}{2}\right). 
	\end{align*}
	Since $|\nabla A|^2-|\nabla|A||^2\geq0$, using the hypothesis $\textup{(a)}$ and the equality above, we conclude that
		\begin{align}\label{l-s-e-gen}
		    |\nabla A|=|\nabla |A||.
		\end{align}
		Now again,  as in the proof of Theorem \ref{huiss}, \eqref{l-s-e-gen} implies two possibilities: (I) If rank of $A$ is greater than 2, $\nabla A=0$.  By the result of Lawson in \cite{lawson}, it follows that $U$ must be  contained in $\mathbb{S}^l_r(0)\times\mathbb{R}^{n-l}$, $2\leq l\leq n$.  (II) If rank of $A$ is  1, then 
		\begin{align*}
		    H^2=|A|^2=c(H-\lambda)^2.
		\end{align*}
		This implies that $H$ and $|A|$ are constants. In particular, by \eqref{l-s-e-gen}, $\nabla A=0$. By the result of Lawson in \cite{lawson}, $U$ must be  contained in  $\mathbb{S}^1_r(0)\times\mathbb{R}^{n-1}$. \\
		Therefore, we have proved that $\Sigma$ is   locally contained in either a hyperplane or  a $\mathbb{S}^l_r(0)\times\mathbb{R}^{n-l}$, $1\leq l\leq n$. Applying the unique continuation properties for solutions of elliptic equations, we conclude that $\Sigma$ is either  a hyperplane or a generalized cylinder $\mathbb{S}^l_r(0)\times\mathbb{R}^{n-l}$, $1\leq l\leq n$.
\end{proof}

  Now we  prove a result (Proposition \ref{prop-1}) on  the   immersed  complete   $\lambda$-self-expanders   with constant mean curvature. 
In \cite[Corollary 3.2]{ancari2020rigidity}, the first author of the present article and Miranda proved that 
a complete $\lambda$-hypersurface is either a hyperplane or a cylinder   $\mathbb{S}^l_r(0)\times \mathbb{R}^{n-l}$, $1\leq l\leq n$. Since a $\lambda$-self-expander with constant mean curvature is also a $(2H-\lambda)$-hypersurface,  Proposition \ref{prop-1} can  be obtained by Corollary 3.2 in \cite{ancari2020rigidity}. In the following, we give a direct proof for $\lambda$-self-expanders for the sake of completeness.

\begin{proposition}\label{prop-1}
	If $\Sigma$ is a complete  $\lambda$-self-expander  immersed  in $\mathbb{R}^{n+1}$ with constant mean curvature, then $\Sigma$  must be either a hyperplane or a cylinder   $\mathbb{S}^l_r(0)\times \mathbb{R}^{n-l}$, $1\leq l\leq n$.
	
\end{proposition}

\begin{proof}
If  $H=0$,  \eqref{e2} implies that  $\lambda|A|^2=0$. Then $\lambda=0$ or $|A|^2\equiv0$ on $\Sigma$.  If $\lambda=0$, $\Sigma$ is a self-expander and must be a hyperplane through the origin. If $|A|^2=0$,  $\Sigma$ also be a hyperplane. \\

 If $H\neq 0$, it follows  that  from   \eqref{e2},
	\begin{align}\label{e9}
		-\frac{1}{2}H-|A|^2(H-\lambda)=0.
	\end{align}
	This implies that  $H\neq \lambda$.  Then $|A|$ is  constant and $\lambda\neq0$. From \eqref{e3},  we get 
	\begin{align}\label{e11}
		|\nabla A|^2=\frac{|A|^2}{2}+\lambda\textup{tr}A^3+|A|^4.
	\end{align}
	On the other hand, by the following   Simons' equation
	\begin{align*}
		\frac{1}{2}\Delta|A|^2= |\nabla A|^2-\left\langle A,Hess(H)\right\rangle-H\textup{tr}A^3-|A|^4,
	\end{align*}
	it follows that 
	\begin{align}\label{e12}
		|\nabla A|^2=|A|^4+H\textup{tr}A^3.
	\end{align}
	 Then, by \eqref{e11} and \eqref{e12}, we have that
	\begin{align}\nonumber
		|\nabla A|^2=|A|^4+\frac{|A|^2H}{2(H-\lambda)}.
	\end{align}
	This and \eqref{e9} imply that 
	\begin{align}\nonumber
		|\nabla A|=0.
	\end{align}
	So, by a result of Lawson in \cite[Theorem 4]{lawson},  $\Sigma$  must be a generalized cylinder. 

\end{proof}

\section{Further discussion}\label{section7}

In this section, we first prove the following result which may be compared with  Theorem \ref{le-s6}.

\begin{theorem}\label{th-71}
Let $\Sigma$ be a complete immersed  $\lambda$-self-expander  in $\mathbb{R}^{n+1}$. Assume that $H^{\delta}\in L^{2}(\Sigma,e^{\frac{|x|^2}{4}}d\sigma)$ for some  $\delta \in\mathbb{N}$ and 
\begin{align}\label{l-s-e-cao}
    |A|^2((2H-\lambda)^2+2n-\lambda^2)\leq0.
\end{align}
Then  $\Sigma$ is either a hyperplane or a sphere $\mathbb{S}^n_r(0)$.
\end{theorem}
\begin{proof}
From hypothesis \eqref{l-s-e-cao} and  \eqref{e2} it follows that
\begin{align*}
    H\mathscr{L}H&=-|A|^2(H-\lambda)H-\frac{H^2}{2}\\
    &\geq-|A|^2(H-\lambda)H-\frac{n|A|^2}{2}\\
                       &=-\frac{1}{4}|A|^2(4(H-\lambda)H+2n)\\
                     & = -\frac{1}{4}|A|^2((2H-\lambda)^2-\lambda^2+2n)\\
                     &\geq0.                            
\end{align*}
	Let $\varphi\in C^{\infty}_0(\Sigma)$. Integrating by part, we have
	\begin{align*}
	\int_{\Sigma}H^{2\delta-2}\varphi^2H(\mathscr{L} H)e^{\frac{|x|^2}{4}}
	 =&-2\int_{\Sigma}H^{2\delta-1}\varphi\langle \nabla\varphi,\nabla H\rangle e^{\frac{|x|^2}{4}}\\
	&-(2\delta-1)\int_{\Sigma}\varphi^2H^{2(\delta-1)}|\nabla H|^2e^{\frac{|x|^2}{4}}\\
	\leq&-\frac{2\delta-1}{2}\int_{\Sigma}\varphi^2H^{2(\delta-1)}|\nabla H|^2e^{\frac{|x|^2}{4}}\\
	&+\frac{2}{2\delta-1}\int_{\Sigma}|\nabla\varphi|^2H^{2\delta}e^{\frac{|x|^2}{4}},
	\end{align*}
	where $\delta\geq1$. Therefore 
	\begin{align*}
	    \frac{2\delta-1}{2}\int_{\Sigma}\varphi^2H^{2(\delta-1)}|\nabla H|^2e^{\frac{|x|^2}{4}}\leq\frac{2}{2\delta-1}\int_{\Sigma}|\nabla\varphi|^2H^{2\delta}e^{\frac{|x|^2}{4}}
	\end{align*}

	Choose  $\varphi=\varphi_j$, where  $\varphi_j$ are  nonnegative cut-off functions satisfying that  $\varphi_j = 1$ on $B_j^{\Sigma}(p)$, $\varphi_j = 0$ on $\Sigma\setminus B_{2j}^{\Sigma}(p)$ and $|\nabla \varphi_j|\leq \frac{1}{j}$. By the monotone convergence theorem and the hypothesis that $H^{\delta}\in L^{2}(\Sigma,e^{\frac{|x|^2}{4}}d\sigma)$ , it follows that, on $\Sigma$,
	$$H^{2(\delta-1)}|\nabla H|^2=0.$$
	This implies that $H$ is constant. If $H=0$, $\Sigma$ is a hyperplane. If $H\neq0$, by Proposition \ref{le-s6},   $\Sigma$ is  at most a cylinder   $\mathbb{S}^l_r(0)\times \mathbb{R}^{n-l}$, $1\leq l\leq n$. However, the integrability hypothesis of $H^{\delta}$ implies that  $ \Sigma$ must be  a sphere.
\end{proof}

In \cite{ancari2020volum},  by using properties on the finiteness of weighted areas and the area growth upper estimate for self-expanders with some restriction on mean curvature, we  proved some theorems that characterize the hyperplanes through the origin as self-expanders  (\cite[Theorems 1.3 and 1.4]{ancari2020volum}). Here, motivated by the work \cite{ancari2020volum}  we will prove  some results that  characterize the generalized cylinders as $\lambda$-self-expanders. In order to prove these results, we first prove the following Theorem \ref{le-ineq5}, which deals with finitiness of weighted areas and area growth estimate of $\lambda$-self-expanders.

\begin{theorem}\label{le-ineq5}
	Let $\Sigma$ be a complete  properly immersed   $\lambda$-self-expander in $\mathbb{R}^{n+1}$.  Assume  that  its mean curvature  $H$ satisfies  $|H|(x)\leq a|x|+b$, $x\in \Sigma$, for  some constants $0\leq a<\frac{1}{2}$ and  $b>0$. Then it holds that, for any  $\alpha >\frac{4a^2}{1-4a^2},$
	\begin{itemize}
		\item[(i)] $\int_{\Sigma}e^{-\frac{\alpha}{4}|x|^2}d\sigma<\infty.$
		\item[(ii)] The area of $B_r(0)\cap\Sigma$ satisfies 
		\begin{align*}
			Area(B_r(0)\cap\Sigma)\leq C(\alpha)e^{\frac{\alpha }{4}r^2},
		\end{align*}
	\end{itemize}
	where $B_r(0)$ denotes  the  round ball in $\mathbb{R}^{n+1}$ of radius $r$ centered at the origin  $0\in \mathbb{R}^{n+1}$. 
	
	In particular,  if  $0\leq a<\frac{1}{2\sqrt{2}}$,  then  the Gaussian weighted area is finite, that is,
	\begin{align}\label{ps}
		\int_{\Sigma}e^{-\frac{|x|^2}{4}}d\sigma<\infty.
	\end{align} 
	
\end{theorem}
\begin{proof}
	Take $h=\frac{|x|^2}{4}$, $x\in\mathbb{R}^{n+1}$. Since $\Sigma$ is properly immersed in $\mathbb{R}^{n+1}$, $h$ is proper on $\Sigma$. Besides, since $H=\lambda-\frac{\langle x,{\bf n}\rangle}{2}$, by \eqref{le-comp} we have that, on $B_r(0)\cap\Sigma$,
	\begin{align*}
		\Delta h=&-H\frac{\langle x,{\bf n}\rangle}{2}+\frac{n}{2}\\
		=&H(H-\lambda)+\frac{n}{2}\\
		\leq&ar^2+(2ab+a|\lambda|)r+(b^2+|\lambda|b+\frac{n}{2}).
	\end{align*}
	In the above, we also has used  the hypothesis: $|H|(x)\leq a|x|+b$, $x\in \Sigma$. We also have that on  $B_r(0)\cap\Sigma$,
	\begin{align*}
		\Delta h-\alpha|\nabla h|^2+\alpha h=&H(H-\lambda)+\frac{n}{2}+\alpha\frac{\langle x,{\bf n}\rangle^2}{4}\\
		=&H(H-\lambda)+\frac{n}{2}+\alpha(H-\lambda)^2\\
		=&(1+\alpha)H^2-\lambda(1+2\alpha)H+\frac{n}{2}+\alpha\lambda^2\\
		\leq&(1+\alpha)a^2r^2+(2ab(1+\alpha)+a|\lambda|(1+2\alpha))r\\
		&+((1+\alpha)b^2+|\lambda|(1+2\alpha)b+\frac{n}{2}+\alpha\lambda^2).
	\end{align*}
	Let $a_2=(1+\alpha)a^2$, $a_1=2ab(1+\alpha)+a|\lambda|(1+2\alpha)$, $a_0=(1+\alpha)b^2+|\lambda|(1+2\alpha)b+\frac{n}{2}+\alpha\lambda^2$ and $\beta=\alpha$.  For  $\alpha >\frac{4a^2}{1-4a^2}$,  where $0\leq a<\frac12$,  it holds that $a_2<\frac{\beta }{4}$. By applying Theorem 3.1 in \cite{ancari2020volum}, we obtain that $\int_{\Sigma}e^{-\frac{\alpha}{4}|x|^2}d\sigma<\infty$ for all  $\alpha >\frac{4a^2}{1-4a^2}$ and the area of $B_r(0)\cap\Sigma$ satisfies that, for all $r>0,$
	\begin{align*}
		Area(B_r(0)\cap\Sigma)\leq C(\alpha)e^{\frac{\alpha }{4}r^2}.
	\end{align*}
	In the particular case of  $a<\frac{1}{2\sqrt{2}}$, since $a<\frac{1}{2\sqrt{2}}$ implies that  $\frac{4a^2}{1-4a^2}<1$,  we may  take $\alpha=1$.
\end{proof}

By an argument analogous to  the ones used in the proofs of Theorem 4.1 in \cite{cheng2013volume} and  Theorem 4 in \cite{ancari2019volume}, we may prove the following  result: Let   $\Sigma$ be a complete  immersed $\lambda$-self-expander in $\mathbb{R}^{n+1}$. If there exists $\alpha>0$ such that   $\int_{\Sigma}e^{-\frac{\alpha}{4}|x|^2}d\sigma<\infty$, then  $\Sigma$ is properly immersed on $\mathbb{R}^{n+1}$.   
Hence Theorem \ref{le-ineq5} has the following consequence.
\begin{cor}
	Let $\Sigma$ be a complete  immersed  $\lambda$-self-expander e in $\mathbb{R}^{n+1}$. Assume that its mean curvature  $ H$  satisfies $|H|(x)\leq a|x|+b$, $x\in\Sigma$, for some constants $0\leq a<\frac{1}{2}$ and $b>0$. Then for $\alpha>\frac{4a^2}{1-4a^2}$  the following statements are equivalent:
	\begin{itemize}
		\item [(i)] $\Sigma$ is properly immersed on $\mathbb{R}^{n+1}$.
		\item [(ii)] There exist constants $C=C(\alpha)$, $\overline{a}_0$, $\overline{a}_1$ and $\overline{a}_2<\frac{\alpha}{4}$, such that 
		\begin{align*}
			Area(B_r(0)\cap\Sigma)\leq C(\alpha)e^{\overline{a}_2r^2+\overline{a}_1r+\overline{a}_0}.
		\end{align*}
		\item [(iii)] $\int_{\Sigma}e^{-\frac{\alpha}{4}|x|^2}d\sigma<\infty.$
	\end{itemize}
\end{cor}
Another consequence of Theorem  \ref{le-ineq5} is the following  corollary, which deals of  the integrable property of the powers of the norm of mean curvature $ H $.
\begin{cor}\label{le-spn10}
	Let $\Sigma$ be  a complete properly immersed $\lambda$-self-expander in $\mathbb{R}^{n+1}$. Assume that $|H|\leq a|x|+b$, $x\in \Sigma,$  for some constants $0\leq a<\frac{1}{2}$ and $b>0$.  Then for $\delta\geq0$ and $\alpha >\frac{4a^2}{1-4a^2}$, 
	\begin{align}
		\int_{\Sigma}|H|^\delta e^{-\frac{\alpha}{4}|x|^2}d\sigma<\infty.
	\end{align}  
\end{cor}
\begin{proof}
	Since $|H|\leq a|x|+b$, it is easily to see that Theorem \ref{le-ineq5} implies the desired conclusion.
\end{proof}

\begin{theorem}\label{le-e18}
	Let $\Sigma$ be a complete  properly immersed $\lambda$-self-expander in $\mathbb{R}^{n+1}$.  Assume  that  its  mean curvature $H$ satisfies $|H|(x)\leq a|x|+b$, $x\in\Sigma$, for  some constants $0\leq a<\frac{1}{2}$ and   $b>0$. If   there exists $\alpha>\frac{4a^2}{1-4a^2}$  such that
	\begin{align}\label{le-eq-cyl-1}
		|A|^2(H-\lambda)H+\frac{H^2}{2}+\frac{\alpha+1}{4}A(x^{\top},x^{\top})H\leq0,
	\end{align}
	
	then $\Sigma$  must be  either a hyperplane or a cylinder   $\mathbb{S}^l_r(0)\times \mathbb{R}^{n-l}$, $1\leq l\leq n$.
	
\end{theorem}
\begin{proof} 
	Let $\varphi\in C^{\infty}_0(\Sigma)$. From \eqref{le-spn8} and hypothesis \eqref{le-eq-cyl-1},  we have
	\begin{align*}
		0&\leq\int_{\Sigma}\left(-|A|^2(H-\lambda)H-\frac{1}{2}H^2-\frac{\alpha+1}{4}A(x^{\top},x^{\top})H \right)\varphi^2e^{-\frac{\alpha}{4}|x|^2}\\
		&=\int_{\Sigma}H\varphi^2(\mathcal{L}_\alpha H)e^{-\frac{\alpha}{4}|x|^2}.
	\end{align*}
	Further 
	\begin{align*}
		\int_{\Sigma}H\varphi^2(\mathcal{L}_\alpha H)e^{-\frac{\alpha}{4}|x|^2}	&=-2\int_{\Sigma}H\varphi\langle \nabla\varphi,\nabla H\rangle e^{-\frac{\alpha}{4}|x|^2}-\int_{\Sigma}\varphi^2|\nabla H|^2e^{-\frac{\alpha}{4}|x|^2}\\
		&\leq-\frac{1}{2}\int_{\Sigma}\varphi^2|\nabla H|^2e^{-\frac{\alpha}{4}|x|^2}+2\int_{\Sigma}|\nabla\varphi|^2H^{2}e^{-\frac{\alpha}{4}|x|^2}.
	\end{align*}
	Therefore 
	\begin{align*}
		\frac{1}{2}\int_{\Sigma}\varphi^2|\nabla H|^2e^{-\frac{\alpha}{4}|x|^2}\leq2\int_{\Sigma}|\nabla\varphi|^2H^{2}e^{-\frac{\alpha}{4}|x|^2}.
	\end{align*}
	
	Choose  $\varphi=\varphi_j$, where  $\varphi_j$ are  nonnegative cut-off functions satisfying  $\varphi_j = 1$ on $B_j(0)$, $\varphi_j = 0$ on $\Sigma\setminus B_{j+1}(0)$ and $|\nabla \varphi_j|\leq 1$.  By the hypothesis $|H|(x)\leq a|x|+b$, $x\in\Sigma$, for some constants $0\leq a<\frac{1}{2}$ and $b>0$, Corollary \ref{le-spn10} implies that $\int_{\Sigma}H^2e^{-\frac{\alpha}{4}|x|^2}<\infty$, for any $\alpha>\frac{4a^2}{1-4a^2}$. By the monotone convergence theorem  it follows that
	$$|\nabla H|=0,$$
	which implies that $H$ is constant.  By Proposition \ref{prop-1}, we get that $\Sigma$ must be a generalized cylinder.
	\end{proof}
%%%%%%%%%%%%%%%%%%%%%%%%%%%%%%%%%%%%%%%%%%%%%%%%%%%%%%%%%%%%
%%%%%%%%%%%%%%%%%%%%%%%%%%%%%%%%%%%%%%%%%%%%%%%%%%%%%%%%%%%%%%

\begin{theorem}\label{th-74}
	Let $\Sigma$ be a complete properly immersed  $\lambda$-self-expander in $\mathbb{R}^{n+1}$. Assume that its mean curvature $H$ is bounded from below and satisfies  $H(x)\leq a|x|+b$, $x\in\Sigma$, for  some constants $0\leq a<\frac{1}{2}$ and $b>0$. If   there exists $\alpha>\frac{4a^2}{1-4a^2}$ such that 
	\begin{align}\label{le-cyl-2}
		|A|^2(H-\lambda)+\frac{H}{2} +\frac{\alpha+1}{4}A(x^{\top},x^{\top})\geq0,
	\end{align}
	
	then $\Sigma$  must be either a hyperplane or a cylinder   $\mathbb{S}^l_r(0)\times \mathbb{R}^{n-l}$, $1\leq l\leq n$.
		
\end{theorem}

\begin{proof}
	Let us fix $C=\inf_{x\in\Sigma}H$. From hypothesis \eqref{le-cyl-2} and \eqref{le-spn8} it follows that
	\begin{align}\label{si1}
		\mathcal{L}_{\alpha}(H-C)\leq0.
	\end{align}
	By  the maximum principle, either $H\equiv C$ or $H> C.$ 
	If $H\equiv C$, by Proposition  \ref{le-s6}, $\Sigma$ is either  hyperplane  or a generalized cylinder. 
	If $H>C$, let us consider $u:=\log(H-C)$.  A computing yields
	\begin{align}\label{e14}
		\Delta u=-|\nabla u|^2+\frac{\Delta H}{H-C}.
	\end{align}     
	Combining   \eqref{si1} and \eqref{e14}, we get
	\begin{align}\label{e16}
		\mathcal{L}_{\alpha} u\leq-|\nabla u|^2.
	\end{align}

	Let us consider the sequence $\varphi_j$ of nonnegative cut-off function satisfying that  $\varphi_j=1$ on $B_j(0)$, $\varphi_j=0$ on $\Sigma\setminus B_{j+1}(0)$ and $|\nabla \varphi_j|\leq 1.$
	Multiplying \eqref{e16} by $\varphi^2_j$ and integrating by parts we obtain 
	\begin{align*}
		\int_{\Sigma}\varphi^2_j|\nabla u|^2e^{-\alpha\frac{|x|^2}{4}}&\leq -\int_{\Sigma}\varphi^2_j(\mathcal{L}_{\alpha} u)e^{-\alpha\frac{|x|^2}{4}}\\
		&=\int_{\Sigma}2\varphi_j\langle \nabla \varphi_j,\nabla u\rangle e^{-\alpha\frac{|x|^2}{4}}\\
		&\leq\frac{1}{2}\int_{\Sigma}\varphi^2_j|\nabla u|^2e^{-\alpha\frac{|x|^2}{4}}+2\int_{\Sigma}|\nabla\varphi_j|^2e^{-\alpha\frac{|x|^2}{4}}.
	\end{align*} 
	Therefore
	\begin{align}\label{le-e17}
		\int_{\Sigma}\varphi^2_j|\nabla u|^2e^{-\alpha\frac{|x|^2}{4}}\leq 4\int_{\Sigma}|\nabla\varphi_j|^2e^{-\alpha\frac{|x|^2}{4}}.
	\end{align}
	From assumption on the mean curvature of $\Sigma$, it follows that $|H|(x)\leq a|x|+b_1$, for some constants $0\leq a<\frac{1}{2}$  and $b_1>0$. Therefore, Theorem  \ref{le-ineq5} implies that $\int_{\Sigma}e^{-\alpha\frac{|x|^2}{4}} <\infty$. Applying the dominated convergence theorem in \eqref{le-e17}, we conclude that 
	\begin{align*}
		\int_{\Sigma}|\nabla u|^2e^{-\alpha\frac{|x|^2}{4}}=0.
	\end{align*}
	In particular   $H$ must be a constant, but this contradicts  the assumption that $H>\inf_{x\in\Sigma}H$.
\end{proof}

%	\bibliographystyle{plain}
%\bibliography{cwmcex}
\end{document}